\documentclass[a4paper,11pt]{amsart}

\usepackage{amsmath,amssymb,amsthm,graphicx,stmaryrd,bbm,MnSymbol,mathtools}
\usepackage{verbatim,enumitem,xcolor}
\usepackage[enableskew]{youngtab}
\usepackage{tikz,tikz-cd}
\usepackage{hyperref}

\newcommand{\CC}{\mathbb{C}}
\newcommand{\EE}{\mathbb{E}} 
\newcommand{\FF}{\mathbb{F}}

\newcommand{\NN}{\mathbb{N}}
\newcommand{\PP}{\mathbb{P}} 
\newcommand{\bP}{\mathbf{P}}

\newcommand{\VV}{\mathbb{V}}

\newcommand{\ZZ}{\mathbb{Z}}

\newcommand{\cA}{\mathcal{A}}
 
\newcommand{\cI}{\mathcal{I}}

\newcommand{\cL}{\mathcal{L}}

\newcommand{\D}{\Delta} 
\renewcommand{\d}{\delta}

\renewcommand{\L}{\Lambda}

\renewcommand{\b}{\beta} 
 
\newcommand{\Om}{\Omega}
\newcommand{\om}{\omega} 
 
\newcommand{\s}{\sigma}

\newcommand{\eps}{\varepsilon}

\newcommand{\tr}{\mathrm{tr}}

\newcommand{\sd}{\triangle}

\renewcommand{\b}{\beta}
\newcommand{\oo}{\infty}

\newcommand{\bra}[1]{\langle#1|}
\newcommand{\ket}[1]{|#1\rangle}

\newcommand{\sm}{\setminus}

\newcommand{\es}{\varnothing}
\newcommand{\se}{\subseteq}

\newcommand{\RN}[1]{%
  \texttt{\uppercase\expandafter{\romannumeral#1}}%
}

\newcommand{\one}{\hbox{\rm 1\kern-.27em I}}

\newcommand{\be}{\begin{equation}}
\newcommand{\ee}{\end{equation}}

\newcommand{\sone}{\sigma^{\scriptscriptstyle (1)}}
\newcommand{\stwo}{\sigma^{\scriptscriptstyle (2)}}
\newcommand{\sthree}{\sigma^{\scriptscriptstyle (3)}}
\newcommand{\eight}{\mathrm{8vx}}

\newcommand{\rx}{\mathrm{x}}
\newcommand{\rz}{\mathrm{z}}
\newcommand{\bJ}{\mathbf{J}}

\newcommand{\dw}{\textsc{dw}}
\newcommand{\EF}{\mathrm{EF}}

\def\1{{\mathchoice {1\mskip-4mu\mathrm l}      
{1\mskip-4mu\mathrm l} 
{1\mskip-4.5mu\mathrm l} {1\mskip-5mu\mathrm l}}} 

\newtheoremstyle{slthm}
{}
{\baselineskip}
{\slshape}
{\parindent}
{\scshape}
{.}
{ }
{}

\theoremstyle{slthm}
\newtheorem{definition}{Definition}[section]
\newtheorem{theorem}[definition]{Theorem}
\newtheorem{proposition}[definition]{Proposition}
\newtheorem{lemma}[definition]{Lemma}

\newtheorem{remark}[definition]{Remark}

\title
[The uniform 8-vertex and
toric code models]
{%
Correlation inequalities for the uniform 8-vertex model and  
the toric code model
}
\author{J. E. Bj\"ornberg}
\thanks{JEB: University of Gothenburg and Chalmers University of
  Technology, Sweden}
\author{B. Lees}
\thanks{BL: Heilbronn Institute for Mathematical Research and School of Mathematics, University of Bristol}
\date{\today}

\begin{document}

\begin{abstract}
We elucidate connections between four models in statistical physics
and probability theory: (1) the toric code model of Kitaev, (2) the uniform
eight-vertex model, (3) random walk on a hypercube, and (4) a 
classical Ising model with four-body interaction.  As a consequence of
our analysis (and of the GKS-inequalities for the Ising model) we
obtain correlation inequalities for the toric code model and the
uniform eight-vertex model.
\end{abstract}

\maketitle


\section{Introduction}
Recent years has seen a rapid development of research on vertex
models, primarily the six-vertex model.  
The six-vertex model was originally introduced by Pauling as a simple
model of hydrogen bonding in water ice.  Lieb
\cite{Lieb671,Lieb672,Lieb673} carried out pioneering work on the
model using the Bethe Ansatz originally developed for the Heisenberg
spin chain.  Recently, the model has received growing attention from
the probability community, with several rigorous results as a
consequence, e.g.\ \cite{DKMO,DKMO2,GP}.

The eight-vertex model is a generalisation of the six-vertex model
that allows two extra types of vertices:
\emph{sources} and \emph{sinks}. 
Both the six- and eight-vertex models are
integrable lattice models \cite{FW,S}.
Using methods coming from integrability, Baxter \cite{Bax} 
computed  the free energy per site.  
For information on integrable models we
direct the reader to the book by Baxter \cite{Baxbook}. 

In this work we consider the eight-vertex model from a different
perspective than solvability/integrability.  We consider some simple
dynamics which preserve eight-vertex configurations:  select a
\emph{plaquette} (face) of the square lattice 
at random and reverse the
direction of its four bounding edges. Because every vertex has an even
number (either zero or two) of its incident edges reversed, these
dynamics preserve the set of allowed configurations. 
(A variant of the dynamics preserves the six-vertex model and is in
fact in detailed balance with the 
six-vertex model's probability distribution \cite{AR}.)

There is a close connection between these dynamics 
and Kitaev's \emph{toric code model}. 
Kitaev's model (at zero temperature) is an important example of a
\emph{quantum code},
which are motivated by problems in quantum computing. 
While quantum computers would allow certain
computations to be performed 
much faster than a conventional computer
(e.g.\ factoring large numbers 
\cite{Shor} or searching unstructured databases
\cite{Grover1,Grover2}), 
a major barrier in realising this potential
is the proclivity for the
quantum bits (qubits) of the computer to to have errors. One way to try
to overcome  
this problem is to use multiple physical qubits to encode a single
\emph{logical} qubit that performs better than its individual physical
qubit components, primarily due to the ability to recover the
intended state of the logical qubit 
even after one or more of the component physical qubits has
experienced an error. We direct the reader to \cite{NC} for an accessible
treatment of major topics in quantum information and quantum computation.
 Surface codes, of which Kitaev's toric code is an
example,  are examples of such a scheme that
have enjoyed huge amounts of attention, in part due to their
relatively high tolerance for local errors \cite{Preskilletal}. 
For an introduction to the various aspects of the toric/surface code
and how it operates we direct the reader to the review \cite{FMMC}.

Using the dynamics indicated above facilitates explicit
computations of certain thermodynamic quantities for the
toric code model at positive temperature.  
Guided by these calculations and their
consequences, we establish a connection between certain
correlations in the toric code model and expectations in a 
many-body classical Ising model.  The ground state
of the latter Ising model also gives the
eight-vertex model with uniform vertex weights.  Using well-known
correlation inequalities for the Ising model (GKS) we obtain
correlation inequalities for the uniform eight-vertex model and the
toric code model.  The following is an example of our correlation
inequality for the uniform eight-vertex model:  given any two sets of
vertices, consider the events that they are all sources or sinks.
Then these events are positively correlated, i.e.\ the probability of
both occurring is at least as large as the product of the
probabilities of the two individual events.
The full statement appears in Theorem \ref{thm:corr}.

\subsection{Setting and results}

For $m,n\geq1$, let $V=\{1,\dotsc,m\}\times\{1,\dotsc,n\}\se \ZZ^2$.
We view $V$ as the vertex set of the $m\times n$ torus $\L$.
(It is possible to work with $\L$ as a subset of $\ZZ^2$ with a
boundary;  we do this in Section \ref{sec:various}.)
Let $E$ denote the set of edges of $\L$ and $F$ the set of faces.  
In keeping with common terminology in the area,  faces will
often be referred to as \emph{plaquettes} and denoted $p\in F$,
and vertices referred to as \emph{stars} and denoted $s\in V$.

The first model we consider is the \emph{uniform eight-vertex model},
defined as follows.
Let $\D$ denote the set of assignments
of directions to the edges $E$, meaning that each horizontal edge is
directed either $\rightarrow$ or $\leftarrow$,
and each vertical edge either $\uparrow$ or $\downarrow$.
Elements of $\D$ will sometimes be referred to as
\emph{arrow-configurations}.  
Let $\D_\eight\se\D$ denote the set 
of arrow-configurations such
that the number of arrows pointing towards any vertex 
$s\in V$ is  even.   At each vertex $s\in V$,
there are eight
possible configurations of arrows, depicted in 
Figure \ref{fig:vertices}.
We refer to these eight local configurations by the roman numerals 
$\RN{1},\RN{2},\dotsc,\RN{8}$.
The first six of these are the allowed configurations for the
six-vertex model;  $\RN7$ is called a \emph{sink}
and $\RN8$ a \emph{source}.
We let $\mu(\cdot)$ denote the uniform probability measure on
$\D_\eight$;  this is the probability measure governing the uniform
8-vertex model.

\begin{figure}[thb]\label{fig:vertices}
  \hspace{-20pt}
  \includegraphics[width=1\textwidth, height=0.2\textwidth]{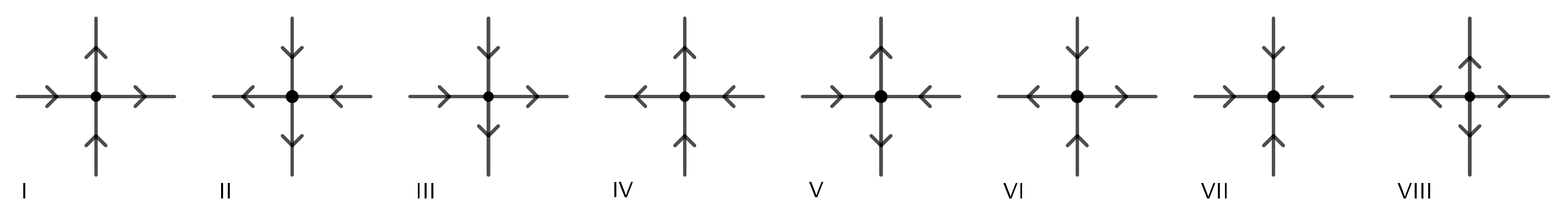}
    \hspace{-20pt}
 \caption{The eight vertex types 
$\RN{1},\RN{2},\dotsc,\RN{8}$
for arrow-configurations    in $\D_\eight$.  
} 
\end{figure}


The next model we consider is Kitaev's \emph{toric code model}
\cite{kitaev}, 
defined as follows.  For $E$ the edge-set of the torus $\L$ as above,
let $(\CC^2)^{\otimes E}$ be the tensor product of one copy of $\CC^2$
for each edge $e\in E$.  
Using the standard basis 
$\ket+=\big(\begin{smallmatrix} 1 \\ 0 \end{smallmatrix}\big)$, 
$\ket-=\big(\begin{smallmatrix} 0 \\ 1 \end{smallmatrix}\big)$
for $\CC^2$, one obtains a (product) basis for $(\CC^2)^{\otimes E}$
with elements $\ket\om$ for $\om\in\Om=\{-1,+1\}^E$.
Consider the Pauli matrices
\be\label{eq:pauli}
\sone=\big(\begin{smallmatrix}
0 & 1 \\ 1 & 0
\end{smallmatrix}\big),
\qquad
\stwo=\big(\begin{smallmatrix}
0 & i \\ -i & 0
\end{smallmatrix}\big),\qquad
\sthree=
\big(\begin{smallmatrix}
1 & 0 \\ 0 & -1
\end{smallmatrix}\big),
\ee
and write $\s^{\scriptscriptstyle(j)}_e$ for the linear 
operator on
$(\CC^2)^{\otimes E}$ which acts as $\s^{\scriptscriptstyle(j)}$
on the $e$ factor and the identity elsewhere.
Since each $\sthree_e$ acts diagonally on the basis elements
$\ket\om$, we refer to this basis as the $\sthree$ product basis. 

Introduce the operators
\be\label{eq:ZX}
Z_s=\prod_{e\sim s} \sthree_e,\qquad
X_p=\prod_{e\sim p} \sone_e,
\qquad s\in V,\quad p\in F,
\ee
where $e\sim s$ and $e\sim p$ mean that the edge $e$ is adjacent to
the vertex $s$ or the face $p$, respectively.
(The operators $Z_s$ and $X_p$ are more commonly denoted
$A_s$ and $B_p$, respectively.)
For $\bJ^\rx=(J^\rx_p:p\in F)$ and 
$\bJ^\rz=(J^\rz_s:s\in V)$ vectors of constants satisfying 
$J^\rx_p,J^\rz_s\geq0$, the \emph{Hamiltonian operator}
is defined as
\be\label{eq:ham}
H=H(\bJ^\rx,\bJ^\rz)=
-\sum_{p\in F} J^\rx_p X_p
-\sum_{s\in V} J^\rz_sZ_s.
\ee
For all $J^\rx_p=J^\rz_s=1$ this is  Kitaev's $H_0$, see \cite{kitaev}.
All terms in $H$ commute, since
for any star and plaquette the number of edges belonging to both
of them is even.  

The Hamiltonian \eqref{eq:ham} governs the toric code model, for
instance through the equilibrium states defined as follows:
for an operator $A$ on $(\CC^{2})^{\otimes E}$,
\begin{equation}\label{eq:eq}
\langle A \rangle=
\langle A\rangle_{\bJ^\rx,\bJ^\rz}=
\frac{\tr [A\,e^{-H}]} {\tr [e^{-H}]}.
\end{equation}
If we take $J^\rx_p=J^\rz_s=\b \to \infty$ we obtain the ground state 
which is important in the theory of quantum codes;  see
Appendix \ref{app:codes}.
We will see below that the limit of all $J^\rx_p\to\oo$,
with all $J^\rz_s=\b$ fixed,  essentially gives
the uniform 8-vertex model.

By further relating the uniform 8-vertex model and the toric code
model  to a certain classical Ising model,
defined in the next subsection, we will prove certain
correlation inequalities.  The setting is as follows.  First, for the
8-vertex model,
let $\eta\in\{\RN{1},\dotsc,\RN{8}\}$  
be any 8-vertex configuration.
Let $\cA^\eta_C$ be the
event that, for each $s\in C$, the arrows around $s$ are either all
equal to those of $\eta$ or all the opposite to those of $\eta$. Additionally,
for $\nu\in\{\RN{1},\dotsc,\RN{8}\}$ with $\nu\neq\eta$ let $\cA^{\eta,\nu}_C$
be the event that, for each $s\in C$, the arrows around $s$ are either all
equal to those of $\eta$ or $\nu$ or all the opposite 
to those of $\eta$ or $\nu$ (this trivially
includes the even $\cA^{\eta}_C$ when $\nu$ has all arrows opposite to
those of $\eta$).
For the toric code the setting is similar but more complex to state.
For $s\in V$, consider the four edges $e_1$, $e_2$, $e_3$ and $e_4$
adjacent to  $s$ and write 
$\sthree_{i}(s)$ for $\sthree_{e_i}$.  Let
$\eps\in\{-1,+1\}^4$, satisfying
$\prod_{i=1}^4 \eps_i=+1$,
be a local sign-configuration.
Define the operators
\be\label{eq:P-op}
P_s^\eps=\tfrac1{16}\prod_{i=1}^4
\big(1+\eps_{i}\sthree_{i}(s)\big),
\qquad
\bar{P}_s^\eps=\tfrac1{16}\prod_{i=1}^4
\big(1-\eps_{i}\sthree_{i}(s)\big),
\ee
and for  $C\se V$ let
\be\label{eq:Q-op}
Q^\eps_C=\prod_{s\in C} ( P^\eps_s+\bar{P}^\eps_s ).
\ee
For $\delta\in\{-1,+1\}^4$ with $\delta\neq \pm\eps$ define
\be
Q^{\eps,\delta}_C=\prod_{s\in C}( P^\eps_s+\bar{P}^\eps_s +P^\delta_s+\bar{P}^\delta_s).
\ee
Finally,
we say that $C\se V$ is \emph{contractible} if $C$ does not contain a
set of nearest-neighbour vertices that forms a non-contractible loop
on the torus $\L$.  

\begin{theorem}
\label{thm:corr}
Let $\eta,\nu$, $\eps,\delta$, with $\nu\neq\eta$ and $\delta\neq\pm\eps$, be as above and let
$C_1,C_2\se V$ be sets of vertices such that
\begin{itemize}[leftmargin=*]
\item either $C_1\cup C_2$ is contractible, 
\item or any non-contractible loop in $C_1\cup C_2$ 
has even length.
\end{itemize}
Then
\begin{enumerate}[leftmargin=*]
\item For the uniform 8-vertex model:
$
\mu(\cA^\eta_{C_1}\cap\cA^\eta_{C_2})\geq
\mu(\cA^\eta_{C_2}) \mu(\cA^\eta_{C_2})
$
and
$
\mu(\cA^{\eta,\nu}_{C_1}\cap\cA^{\eta,\nu}_{C_2})\geq
\mu(\cA^{\eta,\nu}_{C_2}) \mu(\cA^{\eta,\nu}_{C_2}).
$
\item For the toric code model:
$
\langle Q^\eps_{C_1} Q^\eps_{C_2}\rangle \geq
\langle Q^\eps_{C_1} \rangle 
\langle Q^\eps_{C_2}\rangle,
$
and
$
\langle Q^{\eps,\delta}_{C_1} Q^{\eps,\delta}_{C_2}\rangle \geq
\langle Q^{\eps,\delta}_{C_1} \rangle 
\langle Q^{\eps,\delta}_{C_2}\rangle.
$
\end{enumerate}
\end{theorem}

As an example of the correlation inequality for the 8-vertex model, we
can take $\eta$ to be the constant vector of all $\RN{7}$'s, i.e.\
only \emph{sinks}.  Then $\cA^\eta_C$ is the event that every vertex
in $C$ is either a source or a sink, and the Theorem says that 
$\cA^\eta_{C_1}$ and $\cA^\eta_{C_2}$ are positively correlated.  This
is primarily interesting when $C_1$ and $C_2$ are not too far apart:
we will see in Section \ref{ssec:comm} 
that if $C_1$ and $C_2$ are not adjacent
to any common faces, then the states of the vertices in $C_1$
and $C_2$ are in fact independent under $\mu(\cdot)$.

\subsection{Relations between the models}

As mentioned above, Theorem \ref{thm:corr} builds on relating the two
models (uniform 8-vertex and toric code) to a classical Ising model.
We now define the latter.  
Let $\bJ=(J_s:s\in V)$ be a vector with all $J_s\geq 0$.
Recalling that $\Om=\{-1,+1\}^E$,
define the probability measure $\PP_\bJ(\cdot)$
on $\Om$ by  
\be\label{eq:ising-prob}
\PP_\bJ(\s)=\frac{\exp\big(
\textstyle{\sum_{s\in V} J_s \prod_{e\sim s}\s_e}
\big)}
{\sum_{\om\in\Om}\exp\big(
\textstyle{\sum_{s\in V} J_s\prod_{e\sim s}\om_e}
\big)}.
\ee
This is a ferromagnetic Ising model with four-body interaction,
and the restriction of the
toric code model to a certain class of observables is equivalent to 
this model:

\begin{proposition}\label{prop:ising}
Let $Q$ be any observable diagonal in the $\sthree$ product basis,
say $Q\ket\s=q(\s)\ket\s$ for any $\s\in\Om$.
Then 
\be\label{eq:ising}
\langle Q\rangle=\EE_{\bJ^\rz}[q(\s)],
\ee
where $\EE_\bJ[\cdot]$ denotes expectation with respect to
$\PP_\bJ(\cdot)$.  
In particular, $\langle Q\rangle$ does not depend on $\bJ^\rx$.  
\end{proposition}

Note also that if $J_s\to\oo$ for all $s\in V$ then
$\lim\PP_J(\cdot)$ is supported on the set 
$\Om_\eight\se\Om$ of configurations
$\s$ satisfying $\prod_{e\sim s}\s_1=+1$ for all $s\in V$,
since $\sum_{s\in V}\prod_{e\sim s}\s_e$
attains its maximum value $|V|$ for such $\s$.
This observation will allow us to (essentially) identify the uniform
8-vertex measure $\mu(\cdot)$ with 
$\lim\PP_J(\cdot)$.

The key fact we use about the Ising model 
\eqref{eq:ising-prob} is that
it satisfies GKS-inequalities (see e.g.\
\cite[Theorem~3.49]{FandV}):
for any sets $A,B\se E$ we have
\be\label{eq:gks}
\EE_\bJ\big[ \prod_{e\in A} \s_e
\prod_{e\in B} \s_e\big]
\geq \EE_\bJ\big[ \prod_{e\in A} \s_e
\big] \EE_\bJ\big[
\prod_{e\in A} \s_e
\big],\qquad
\EE_\bJ \big[\prod_{e\in A} \s_e
\big]\geq 0.
\ee

The fourth model which we use in our analysis is 
\emph{simple random  walk on the hypercube}.  
To define this, let $\FF_2^+$ denote the two-element group with
elements $\{0,1\}$ satisfying $1+1=0$, and consider 
$G=(\FF_2^+)^F$ with generators given by  
\be\label{eq:gp}
g^p_j=\left\{
\begin{array}{ll}
1, & \mbox{if } j = p,\\
0, & \mbox{otherwise}.
\end{array}
\right.
\ee
Then $G$ is a hypercube of dimension $|F|$.
A random walk on $G$ is given by
selecting, independently at random,  faces $p_1,p_2,\dotsc,p_k$
and letting $X(k)=X(0)+\sum_{i=1}^k g^{p_i}$
where $X(0)$ is a starting position.
Here at each step we choose face $p$ with probability
$J_p/\sum_{q\in F} J_q$.

The relation to the  previous models comes by letting $G$
act on the sets $\Om$ and $\D$.  The intuition is simple:  $g^p$ acts
on elements of 
$\D$ by reversing all the arrows around the face $p$, and on
$\Om$ by negating the $\pm$  signs on the edges surrounding $p$.  It
is easy to see that this action leaves $\D_\eight\se \D$ invariant,
and also leaves the measures $\mu(\cdot)$ and $\PP_\bJ(\cdot)$
invariant.  The connection to the toric code model is slightly more
subtle:  we will see in Section \ref{sec:dynamics} that the term 
$\sum_{p\in  F} J^\rx_p X_p$ in the Hamiltonian \eqref{eq:ham}
is essentially the generator matrix of (the continuous-time version
of)  simple random walk on $G$.  This allows us to express 
the equilibrium states \eqref{eq:eq} for $H$ 
using transition probabilities for the random walk,
which will be instrumental in understanding some basic features of the
toric code model, including the relation to the Ising model,
Proposition \ref{prop:ising}.

\subsection{Acknowledgements}

JEB gratefully acknowledges financial 
support from Vetenskapsr{\aa}det,
grant 2019-04185, from \emph{Ruth och Nils Erik Stenb\"acks
  stiftelse}, and from the
Sabbatical Program at the Faculty of Science,
University of Gothenburg,
as well as kind hospitality at the University of Warwick and the
University of Bristol.

\section{Properties of the toric code model}

\subsection{Dynamic description}
\label{sec:dynamics} 

Recall that (below \eqref{eq:gp})
we defined random walk on the hypercube $G=(\FF_2^+)^F$ as
the process obtained by sampling independently at random
faces $p_1,p_2,\dotsc$ and successively adding them to an initial
element of $G$.  We make this into a continuous-time random walk using
a Poisson process $N(t)$ of rate $|F|$ and defining
\be\label{eq:rw-g}
X(t)=X(0)+\sum_{i=1}^{N(t)} g^{p_i}.
\ee
An equivalent description is that each face $p\in F$ is assigned an
independent exponentially distributed random `clock' 
of rate $J_p$  and that we add 1
in the postion of $p$ when the corresponding clock rings.

We next map the random walk \eqref{eq:rw-g} onto a process in
$\Om=\{-1,+1\}^E$ as well as onto a process in the set $\D$ of
arrow-configurations.  As already described for the discrete-time
case, the process in $\Om$ proceeds by `flipping' (negating) all signs
on the edges around the selected face $p$, and the process in $\D$
similarly reverses the arrows around $p$.  
We may describe this more formally as follows.  The group $G$ acts on
the set $\Om$ by
\be\label{eq:g-s}
g^p:\om\mapsto x^p\cdot\om,
\qquad\mbox{where}\quad
x^p_e=\left\{
\begin{array}{ll}
-1, & \mbox{if } e\sim p,\\
+1, & \mbox{otherwise},
\end{array}
\right.
\ee
and $\cdot$ denotes component-wise multiplication.  
(We are implicitly presenting $\FF_2^+$ as the multiplicative group
with elements $\{-1,+1\}$.)
Then we define $\om(t)$ as the result of applying \eqref{eq:g-s} for
the randomly selected faces $p_i$ in \eqref{eq:rw-g}:
\be\label{eq:rw-s}
\om(t)=\big(\prod_{i=1}^{N(t)} x^{p_i} \big)\cdot \om(0).
\ee
Here $\om(0)\in\Om$ is an initial configuration, and since the group
$G$ is commutative we do not need to specify an order of
multiplication in \eqref{eq:rw-s}.

It will be useful to record here the generator-matrix for the random
walk \eqref{eq:rw-s}.  Indeed, regarding the $x^p$ of \eqref{eq:g-s}
as a matrix with rows and columns indexed by $\Om$,  the generator
matrix can be written as 
\be\label{eq:gen-x}
\sum_{p\in F} J_p (x^p-1),
\ee
where $1$ denotes the identity matrix.  It follows that the
transition-probabilities at time $t$ are given by the matrix
\be
\exp\big(t \textstyle\sum_{p\in F} J_p (x^p-1)\big).
\ee

To formally define the process in $\D$, we note that
 we may think of $\om\in\Om$ as
encoding an arrow-configuration in $\D$
by indicating which edges are reversed
(or not) with respect to some a-priori
configuration of arrows.  More precisely, fix a
reference-configuration $\rho\in\D$ and define, for $\om\in\Om$, the
arrow-configuration $\om\cdot\rho$ such that the arrow at $e$ has the
\emph{same} orientation as in $\rho$ if $\om_e=+1$, respectively the
\emph{opposite} orientation if $\om_e=-1$.  We then obtain a random walk
$\d(t)\in\D$ by
\be \label{eq:rw-d}
\d(t)=\om(t)\cdot\rho.
\ee
Note that $\d(t)$ depends both on the reference-configuration
$\rho\in\D$ and the initial sign-configuration $\om(0)\in\Om$.

We next describe the relevance of the random walk $\om(t)$ for the
toric code model.  Recall that the elements $\om\in\Om$ are in
one-to-one correspondence with basis-vectors $\ket\om$ for 
$(\CC^2)^{\otimes E}$, given as the product-basis obtained from the
basis 
$\ket+=\big(\begin{smallmatrix} 1 \\ 0 \end{smallmatrix}\big)$, 
$\ket-=\big(\begin{smallmatrix} 0 \\ 1 \end{smallmatrix}\big)$
for $\CC^2$.  Thus the action of $G$ on $\Om$ carries over to a
representation of $G$ on $(\CC^2)^{\otimes E}$.  Moreover,
from the explicit form of the Pauli-matrices \eqref{eq:pauli} we see
that $\sone\ket\pm=\ket\mp$.  Thus the operator $X_p$ in \eqref{eq:ZX}
acts precisely as $x^p$ in \eqref{eq:g-s}.  It follows from
\eqref{eq:gen-x} that the term 
$\sum_{p\in F} J^\rx_p X_p$ in the Hamiltonian
\eqref{eq:ham} is, up to adding a multiple of the identity, the
generator for a 
random walk $\ket{\om(t)}$  on the set of basis-vectors.

We note here that the random walk on $\D_\eight$ is
not irreducible, i.e.\ the set decomposes into several disjoint
communicating classes.  See Proposition \ref{prop:CCtorus}.

\subsection{Duality}
\label{sec:duality}

There are two useful notions of duality for the toric code model:
that of the lattice $\L$ as a planar graph, as well as a duality
between the operators $Z_s$ and $X_p$.  In this subsection we aim to
make precise these dualities as well as the connection between them. 

We start by looking at the operators $Z_s$ and $X_p$.  Recall that we
have been using a basis 
for $\CC^2$ in which $\sthree$ is diagonal and the other Pauli
matrices are given as in \eqref{eq:pauli}.  
For the rest of this subsection, we denote this basis by
$\ket{+}_\rz=\big(\begin{smallmatrix} 1 \\ 0 \end{smallmatrix}\big)$, 
$\ket{-}_\rz=\big(\begin{smallmatrix} 0 \\ 1 \end{smallmatrix}\big)$
where the subscript z serves to indicate that the third Pauli matrix is
diagonal.  The corresponding product basis for $(\CC^2)^{\otimes E}$
will be denoted $\ket{\om}_\rz$ for $\om\in\Om$.
Recall that, in this basis, the operator $X_p$ negates the signs on
all edges $e$ surrounding $p\in F$, while $Z_s$ is diagonal.

We can also consider a basis for $\CC^2$ in which $\sone$ is
diagonal:  define
\be
\ket{+}_\rx=\frac{
\ket{+}_\rz+\ket{-}_\rz}{\sqrt 2} 
\qquad \mbox{and} \qquad
\ket{-}_\rx=\frac{
\ket{+}_\rz-\ket{-}_\rz}{\sqrt 2}.
\ee
The corresponding product basis for $(\CC^2)^{\otimes E}$
will be denoted $\ket{\om}_\rx$ for $\om\in\Om$.
The basis-change matrix for going between
$\ket{\pm}_\rz$ and $\ket{\pm}_\rx$ is the symmetric, orthogonal
matrix (the Hadamard matrix)
\be
U=\tfrac1{\sqrt2}
\big(\begin{smallmatrix} 1 & 1 \\ 1 & -1 \end{smallmatrix} \big),
\qquad\mbox{for which}\quad
U\ket{\pm}_\rz=\ket{\pm}_\rx.
\ee
To go between the bases for $(\CC^2)^{\otimes E}$ one uses the
$E$-fold tensor product $U^{\otimes E}$.
This change of basis maps $\sone$ to $U\sone U=\sthree$ and $\sthree$
to $U\sthree U=\sone$, meaning that in the $\ket\cdot_\rx$-basis, 
$Z_s$ negates all the signs on the edges
adjacent $s\in V$, while $X_p$ is diagonal.

What we have described so far dovetails with the planar duality of
$\L$, as follows.  Define the dual $\L^\ast$ to be the graph with
vertex set $V^\ast=F(\L)$, edge-set $E^\ast=E(\L)$, and faces
$F^\ast=V(\L)$.  One obtains $\L^\ast$ by placing a vertex $s^\ast(p)$
in the middle of each face $p$ of $\L$ and drawing edges
perpendicularly across those of $\L$.  In this way, any vertex $s$ of
the orignal lattice $\L$ lies in the middle of a unique face
$p^\ast(s)$ of the dual $\L^\ast$.  We can then write, for $s\in V$
and $p\in F$,
\be\begin{split}
U^{\otimes E}X_pU^{\otimes E}
&=\prod_{e\sim p} U\sone_eU
=\prod_{e\sim s^\ast(p)} \sthree_e 
=:Z_{s^\ast},
\mbox{ and}\\\
U^{\otimes E}Z_sU^{\otimes E}
&=\prod_{e\sim s} U\sthree_eU
=\prod_{e\sim p^\ast(s)} \sone_e
=: X_{p^\ast}.
\end{split}
\ee
As noted above, the operators $Z_{s^\ast}$ and $X_{p^\ast}$ act on the 
$\sone$-basis 
$\ket{\om}_\rx$ in the exact same way as 
$Z_{s}$ and $X_{p}$ act on the $\sthree$-basis $\ket{\om}_\rz$.
Namely, 
$Z_{s^\ast}$ is diagonal while $X_{p^\ast}$ negates the signs around
$p^\ast$.  In this sense, going between the lattice $\L$ and its dual
$\L^\ast$ is equivalent to changing basis between $\ket{\cdot}_\rz$
and $\ket{\cdot}_\rx$.  Some particular instances of this are the
identities:   
\be\label{eq:duality}
\begin{split}
&\prescript{}{\rz}{\bra{\tau}}
\exp\big(t{\textstyle \sum_{p\in F}} J^\rx_p X_p\big)\ket{\s}_\rz
=\prescript{}{\rx}{\bra{\tau}}
\exp\big(t{\textstyle \sum_{s^\ast\in V^\ast}}
J^\rx_{s^\ast} Z_{s^\ast}
\big)\ket{\s}_\rx, \\
&\prescript{}{\rz}{\bra{\tau}}
\exp\big(t{\textstyle \sum_{s\in V}} J^\rz_s Z_s\big)\ket{\s}_\rz
=\prescript{}{\rx}{\bra{\tau}}
\exp\big(t{\textstyle \sum_{p^\ast\in F^\ast}} 
J^\rz_{p^\ast} X_{p^\ast}
\big)\ket{\s}_\rx,
\end{split}
\ee
and for any $A\se V(\L)$, $B\se F(\L)$ with corresponding dual sets 
$A^\ast\se F(\L^\ast)$, $B^\ast\se V(\L^\ast)$, 
\be
\prescript{}{\rz}{\bra{\tau}}
\textstyle 
\prod_{s\in A} Z_s \prod_{p\in B} X_p
\ket{\s}_\rz
=\prescript{}{\rx}{\bra{\tau}}
\textstyle 
\prod_{p^\ast\in A^\ast} X_{p^\ast} 
\prod_{s^\ast\in B^\ast} Z_{s^\ast}
\ket{\s}_\rx.
\ee
These identities may all be obtained by conjugating with 
$U^{\otimes  E}$.

\subsection{Consequences}

We now provide some applications of the random-walk dynamics and of
the duality.  
Write $\bP_\s(\cdot)$ for the probability measure governing the
process $\om(t)$ of \eqref{eq:rw-s} started at $\om(0)=\s$.
We revert to the notation $\ket{\om}$ without a subscript for the
usual basis defined above \eqref{eq:pauli} and referred to as
$\ket{\om}_\rz$ above.

\begin{lemma}\label{lem:dyn}
For any $\s,\tau\in\Om$ and $t>0$,
we have the following matrix entries:
\be\begin{split}\label{eq:mtrx-entries}
&\bra{\tau}\exp\big(t{\textstyle \sum_{p\in F}} J^\rx_p X_p\big)\ket\s
=e^{t\sum_{p\in F} J^\rx_p}\bP_\s(\om(t)=\tau),\\
&\bra{\tau}\exp\big(t {\textstyle \sum_{s\in V}} J^\rz_s Z_s\big)
\ket\s
=
\exp\big(t{\textstyle \sum_{s\in V}
J^\rz_s \prod_{e\sim s} \s_e}\big)
\one_{\s=\tau}.
\end{split}\ee
Moreover,
\be\label{eq:expB-diag}
e^{t\sum_{p\in F} J^\rx_p }\bP_\s(\om(t)=\s)=
\prod_{p\in F} \cosh(tJ^\rx_p)+
\prod_{p\in F} \sinh(tJ^\rx_p).
\ee
In particular, 
$\exp\big(t{\textstyle \sum_{p\in F}} J^\rx_p X_p\big)$ is 
constant on the diagonal.
\end{lemma}
\begin{proof}
The first identity in \eqref{eq:mtrx-entries}
follows from the discussion at the end of Section \ref{sec:dynamics},
since 
$\exp\big(t{\textstyle \sum_{p\in F}} J^\rx_p (X_p-1)\big)$ is the matrix of
transition-probabilities of $\om(t)$.  The second
identity in \eqref{eq:mtrx-entries} follows from the fact that 
$\exp\big(t {\textstyle \sum_{s\in V}} J^\rz_s Z_s\big)$
is diagonal in the given basis, with $(\s,\s)$ entry precisely
$\exp\big(t{\textstyle \sum_{s\in V} J^\rz_s \prod_{e\sim s} \s_e}\big)$.

For \eqref{eq:expB-diag}, we see from 
\eqref{eq:rw-s} that $\om(t)=\om(0)$ if and only if
either (i) each plaquette is flipped an even number of times, or (ii)
each plaquette is flipped an odd number of times.   The number of
times a given plaquette $p$ 
is flipped by time $t$ is a Poisson random
variable with parameter $J_p^\rx t$, and they are independent.  This gives 
the stated expression.
\end{proof}

Lemma \ref{lem:dyn} allows for
straightforward, explicit calculation of some thermodynamic
quantities: 

\begin{proposition}\label{prop:correlations}
The partition function is given explicitly as
\be\label{eq:pf}
\tr [e^{-H}]=
2^{|E|}\big(\prod_{s\in V} \cosh(J^\rz_s)
+\prod_{s\in V} \sinh(J^\rz_s)\big)
\big(\prod_{p\in F}\cosh(J^\rx_p)+
\prod_{p\in F}\sinh(J^\rx_p)\big).
\ee
Moreover,
for any $C\subset F$ and $D\subset V$ we have 
\be\label{eq:corr1}
\Big\langle \prod_{p\in C} X_p\Big\rangle
=\frac{\prod_{p\in F\sm C}\cosh(J^\rx_p)
\prod_{p\in C} \sinh(J^\rx_p)+
\prod_{p\in F\sm C}\sinh(J^\rx_p)
\prod_{p\in C}\cosh(J^\rx_p)}
{\prod_{p\in F}\cosh (J^\rx_p)+
\prod_{p\in F}\sinh (J^\rx_p)}
\ee
and
\be\label{eq:corr2}
\Big\langle \prod_{s\in D} Z_s\Big\rangle
=\frac{\prod_{s\in V\sm D}\cosh(J^\rz_s)
\prod_{s\in D} \sinh(J^\rz_s)+
\prod_{s\in V\sm D}\sinh(J^\rz_s)
\prod_{s\in D}\cosh(J^\rz_s)}
{\prod_{s\in V}\cosh (J^\rz_s)+
\prod_{s\in V}\sinh (J^\rz_s)}.
\ee
In particular, \eqref{eq:corr1} and \eqref{eq:corr2} are both positive.
\end{proposition}

\begin{proof}
For \eqref{eq:pf}, note that 
\be\begin{split}
\tr [e^{-H}]&=\tr\big[
\exp\big({\textstyle \sum_{s\in V}} J^\rz_s Z_s\big)
\exp\big( {\textstyle \sum_{p\in F}} J^\rx_p X_p\big)\big]\\
&=2^{-|E|}\tr\big[
\exp\big( {\textstyle \sum_{s\in V}} J^\rz_s Z_s\big)\big]
\tr\big[
\exp\big( {\textstyle \sum_{p\in F}} J^\rx_p X_p\big)\big].
\end{split}\ee
The first equality follows from the fact that the terms commute, the
second equality from Lemma \ref{lem:dyn} and the elementary fact that
$\tr[AB]=\tr[A]\tr[A]/r$ if $A,B$ are $r\times r$ matrices such that
$A$ is diagonal and $B$ is constant on the
diagonal.  Next, from Lemma \ref{lem:dyn} we have
\be\label{eq:tr-expB}
\tr\big[
\exp\big( {\textstyle \sum_{p\in F}} J^\rx_p X_p\big)\big]=
2^{|E|}\Big(
\prod_{p\in F} \cosh(tJ^\rx_p)+
\prod_{p\in F} \sinh(tJ^\rx_p)\Big).
\ee
Also, from the duality-relations \eqref{eq:duality} we have
\be\label{eq:tr-expA}
\tr\big[
\exp\big( {\textstyle \sum_{s\in F}} J^\rz_s Z_s\big)\big]=
2^{|E|}\Big(
\prod_{s\in V} \cosh(tJ^\rz_s)+
\prod_{s\in V} \sinh(tJ^\rz_s)\Big)
\ee
This follows by taking the trace in the
$\ket{\om}_\rx$-basis rather than the 
$\ket{\om}_\rz$-basis.

For \eqref{eq:corr1}, we note that 
\be\begin{split}\label{eq:corr-exp}
\tr\big[
\big(\textstyle{\prod_{p\in C} X_p}\big) &e^{-H}
\big]
=
\tr\big[
\exp\big(\textstyle{\sum_{s\in V} J^\rz_s Z_s}\big)
\big(\textstyle{\prod_{p\in C} X_p}\big)
\exp\big(\textstyle{\sum_{p\in F} J^\rx_p X_p}\big)
\big]\\
&=\sum_{\s,\tau\in\Om}
\bra{\s}\exp\big(\textstyle{\sum_{s\in V} J^\rz_s Z_s}\big)\ket\tau
\bra\tau \big(\textstyle{\prod_{p\in C} X_p}\big)
\exp\big(\textstyle{\sum_{p\in F} J^\rx_p X_p}\big)\ket\s\\
&=\sum_{\s\in\Om}
\bra{\s}\exp\big(\textstyle{\sum_{s\in V} J^\rz_s Z_s}\big)\ket\s
\bra\s \big(\textstyle{\prod_{p\in C} X_p}\big)
\exp\big(\textstyle{\sum_{p\in F} J^\rx_p X_p}\big)\ket\s,
\end{split}\ee
where we used the fact that 
$\exp\big(\textstyle{\sum_{s\in V} J^\rz_s Z_s}\big)$ is diagonal.  
Next, using the notation of \eqref{eq:g-s},
\be\label{eq:corr-prob}
\bra\s \big(\textstyle{\prod_{p\in C} X_p}\big)
\exp\big(\textstyle{\sum_{p\in F} J^\rx_p X_p}\big)\ket\s
=e^{\sum_{p\in F} J^\rx_p}
\bP_\s\big(\om(\b)=\textstyle{\prod_{p\in C} x^p}\cdot \s\big).
\ee
For the event in the probability to occur, either (i) all $p\in C$ are
flipped an odd number of times and all $p\in F\sm C$ an 
even number of
times, or (ii) vice versa.  Computing this probability 
and using \eqref{eq:tr-expA} and \eqref{eq:pf} 
gives \eqref{eq:corr1}.
Finally, \eqref{eq:corr2} follows from the duality \eqref{eq:duality}.
\end{proof}

From \eqref{eq:pf} we immediately obtain an expression for the
\emph{free energy} of the system 
in the case when all $J^\rx_p=\b_\rx$ and all
$J^\rz_s=\b_\rz$: 
\be
f(\b_\rx,\b_\rz)=\lim_{|V|\to\oo} \tfrac{1}{|V|}
\log\big( \tr[ e^{-H}]\big)
=\log(\cosh\b_\rx)+\log(\cosh \b_\rz)+2\log 2.
\ee
In particular, $f(\b_\rx,\b_\rz)$ is analytic
for $\b_\rx,\b_\rz>0$, indicating that 
the system does not undergo a phase transition, 
as pointed out earlier in e.g.\ \cite{AFH}.

Next we record a result 
 complementary  to Proposition \ref{prop:correlations}
that will be useful in the proof of the
correlation inequalities.

\begin{proposition}\label{prop:othercorrelations}
For $A\subset E$ define the operators 
\be
X_A=\prod_{e\in A}\sone_e,
\qquad 
Z_A=\prod_{e\in A}\sthree_e.
\ee
Then  $\langle X_A\rangle=0$
unless $X_A$ is a product of plaquette operators $X_p$,
i.e.\ unless $X_A =\prod_{p\in C} X_p$ for some $C\se F$.
Similarly, $\langle Z_A\rangle=0$
unless $Z_A$ is a product of star operators $Z_s$,
i.e.\ unless $Z_A =\prod_{s\in D} Z_s$ for some $D\se V$.
\end{proposition}
\begin{proof}
We prove the claim for $X_A$;  the claim for $Z_A$ follows 
by  the duality \eqref{eq:duality}.
The expansion \eqref{eq:corr-exp}
remains valid with $\prod_{p} X_p$ replaced by $X_A$, and in place of
\eqref{eq:corr-prob} we find that
\be
\bra\s  X_A
\exp\big(\textstyle{\sum_{p\in F} J^\rx_p X_p}\big)\ket\s
=e^{\sum_{p\in F} J^\rx_p}
\bP_\s\big(\om(\b)=\textstyle{\prod_{e\in A} x^e}\cdot \s\big),
\ee
where $x^e$ is $-1$ in position $e$ and $+1$ elsewhere.
Now if $X_A$ is not a product of plaquette operators
then there is no realisation of the dynamics
with the property that 
$\om(\b)=\textstyle{\prod_{e\in A} x^e}\cdot \s$.
Then 
$\bP_\s\big(\om(\b)=\textstyle{\prod_{e\in A} x^e}\cdot \s\big)=0$,
as claimed.
\end{proof}

We now turn to the Ising model \eqref{eq:ising-prob}
and  Proposition \ref{prop:ising}.  Recall that $Q$ is 
assumed to be an operator on
$(\CC^2)^{\otimes E}$ diagonal in the
$\ket\cdot=\ket{\cdot}_\rz$-basis
with $Q\ket\s=q(\s)\ket\s$.

\begin{proof}[Proof of Proposition \ref{prop:ising}]
We use Lemma \ref{lem:dyn} and
orthonormality to expand:
\be\begin{split}
\tr[ Q e^{-H}]&=
\sum_{\s,\tau,\varphi\in\Om}
\bra\s Q \ket\tau\bra\tau
\exp\big(\textstyle\sum_{s\in V} J^\rz_s Z_s
\big)\ket\varphi
\bra\varphi
\exp\big(\sum_{p\in F} J^\rx_p X_p
\big)\ket\s\\
&=
\sum_{\s,\tau,\varphi\in\Om}
q(\tau)\delta_{\s,\tau}
\exp\big(
\textstyle\sum_{s\in V} J^\rz_s
\prod_{e\sim s}\varphi_e\big)
\d_{\tau,\varphi}
 e^{\sum_{p\in F} J^\rx_p}
\bP_\s(\om(\b)=\varphi)\\
&=\sum_{\s\in\Om} q(\s) \exp\big(
\textstyle\sum_{s\in V}J^\rx_s \prod_{e\sim s}\s_e\big)
e^{\sum_{p\in F} J^\rx_p}
\bP_\s(\om(\b)=\s).
\end{split}\ee
In particular,
\be
\tr[e^{-H}]=\sum_{\s\in\Om} \exp\big(
\textstyle\sum_{s\in V} J^\rz_s
\prod_{e\sim s}\s_e\big) e^{\sum_{p\in F} J^\rx_p}
\bP_\s(\om(\b)=\s).
\ee
By Lemma \ref{lem:dyn}, 
$e^{\sum_{p\in F}J^\rx_p}\bP_\s(\om(\b)=\s)$ 
does not depend on $\s$,
thus these factors cancel in
$\langle Q\rangle=\tr[ Q e^{-H}]/
\tr[e^{-H}]$,
giving the result.
\end{proof}

\begin{remark}\phantom{hej}
\begin{enumerate}[leftmargin=*]
\item Recall that the probability measure $\PP_\bJ(\cdot)$  is
invariant under reversing all the spin-values around any fixed
plaquette $p$.  This immediately gives that 
$\EE_\bJ[\prod_{e\in A} \s_e]=0$ if there is any plaquette $p$ such
that $A$ 
contains an odd number of the edges surrounding $p$,
which is in line
with Proposition \ref{prop:othercorrelations}.
\item Clearly the analog of Proposition \ref{prop:ising} is true also
  for operators $Q$ diagonal in the $\ket{\om}_\rx$-basis, by the
  duality \eqref{eq:duality}.  Indeed, the same argument carried out
  in that basis gives $\langle Q\rangle=\EE^\ast_{\bJ^\rx}[q(\s)]$, where 
$\EE^\ast_\bJ[\cdot]$ denotes expecation under the measure
\[
\PP^\ast_\bJ(\s)=\frac{\exp\big(
\textstyle{\sum_{p\in F} J_p \prod_{e\sim p}\s_e}
\big)}
{\sum_{\om\in\Om}\exp\big(
\textstyle{\sum_{p\in F}J_p \prod_{e\sim s}\om_e}
\big)},
\qquad \s\in\Om.
\]
In this sense one can see the toric code model as two coupled
classical Ising models.
\end{enumerate}
\end{remark}

\section{Correlation inequalities}

The proof of Theorem \ref{thm:corr} will proceed by first proving a
similar statement for the Ising model \eqref{eq:ising-prob}.  To state
the latter, we introduce the following notation.  For $s\in V$
consider the 4 edges $e_1$, $e_2$, $e_3$ and $e_4$ adjacent to
$s$, to be definite ordered as in Figure \ref{fig:vertexedges}.  
For $\s\in\Om$, write $\s_i(s)=\s_{e_i}$.  We define the following
quantities similar to \eqref{eq:P-op} and \eqref{eq:Q-op}:
for $\eps\in\{-1,+1\}^4$
satisfying $\prod_{i=1}^4 \eps_i=+1$, let
\be
I^\eps_s(\s)=\tfrac1{16}\prod_{i=1}^4 
\big( 1+\eps_i\s_i(s)\big),
\qquad
\bar{I}^\eps_s(\s)=\tfrac1{16}\prod_{i=1}^4 
\big( 1-\eps_i\s_i(s)\big),
\ee
and for $C\se V$,
\be
\cI^\eps_C(\s)=\prod_{s\in C}(I^\eps_s(\s)+\bar{I}^\eps_s(\s)).
\ee
Then $\cI^\eps_C(\s)$ is the indicator of the event that, for each
$s\in C$, the values $\s_i(s)$ for $i=1,\dotsc,4$ either all agree
with $\eps_i$, or are all the opposite. Similarly to the eight-vertex model
and toric code, for $\delta\in\{-1,+1\}^4$ with $\delta\neq \pm \eps$ we also 
define 
\be
\cI^{\eps,\delta}_C(\s)=\prod_{s\in C}(I^\eps_s(\s)+\bar{I}^\eps_s(\s) + I^\delta_s(\s)+\bar{I}^\delta_s(\s)).
\ee

\begin{figure}[b]\label{fig:vertexedges}
  \includegraphics[scale=0.25]{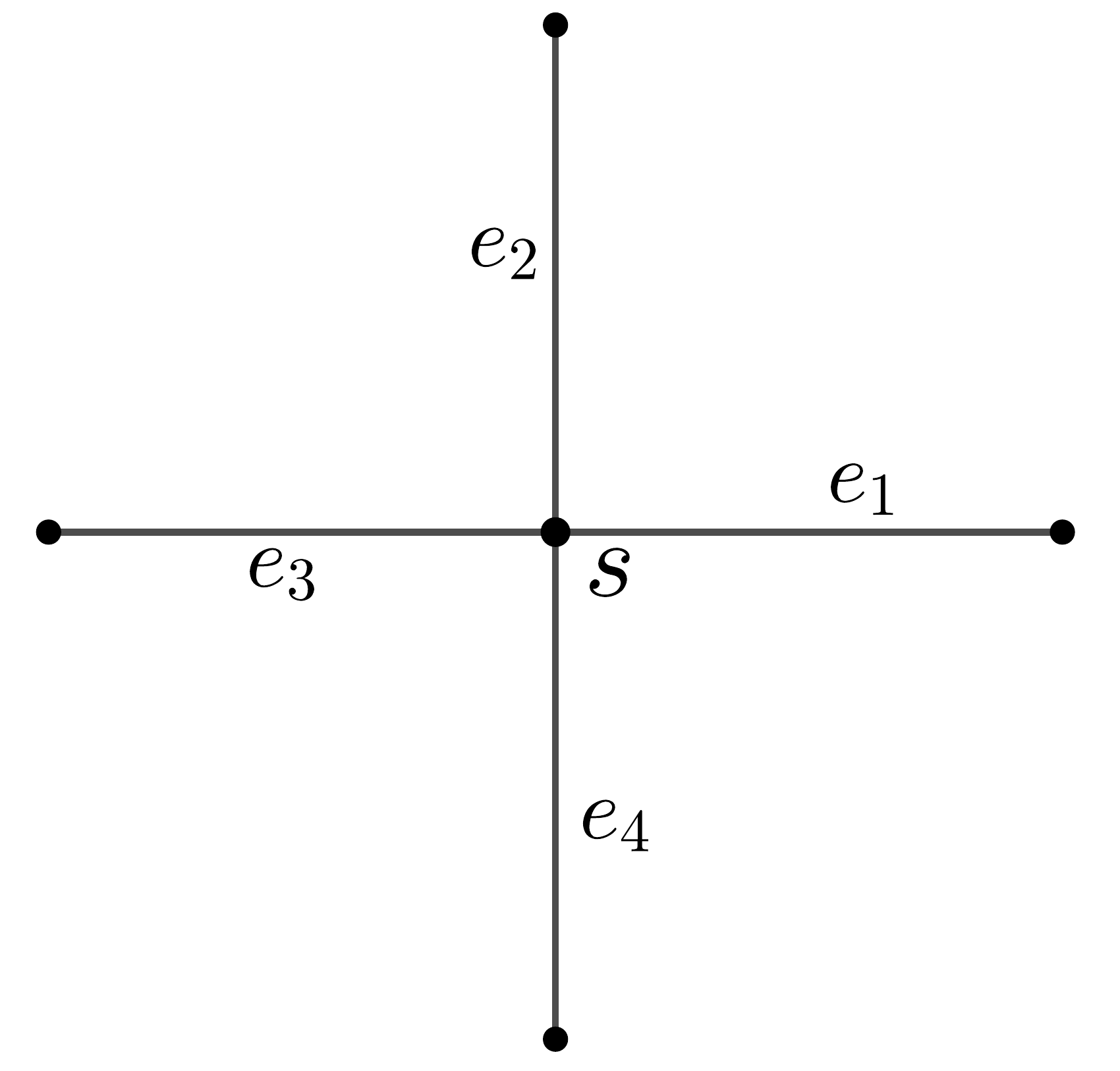}
 \caption{A vertex $s$ with its four incident
 edges labelled in a clockwise fashion.}
\end{figure}

\begin{lemma}\label{lem:corr-ising}
Let $\eps,\delta$, with $\delta\neq \pm \eps$, be as above and let
$C_1,C_2\se V$ be sets of vertices such that
\begin{itemize}[leftmargin=*]
\item either $C_1\cup C_2$ is contractible, 
\item or any non-contractible loop in $C_1\cup C_2$ 
has even length.
\end{itemize}
Then for any $\bJ=(J_s:s\in V)$ such that all  $J_s\geq 0$,
\be \label{eq:corr-ising}
\EE_\bJ[\cI^\eps_{C_1} \cI^\eps_{C_2}]\geq
\EE_\bJ[\cI^\eps_{C_1} ]  \EE_\bJ[ \cI^\eps_{C_2}].
\ee
and
\be
\EE_\bJ[\cI^{\eps,\delta}_{C_1} \cI^{\eps,\delta}_{C_2}]\geq
\EE_\bJ[\cI^{\eps,\delta}_{C_1} ]  \EE_\bJ[ \cI^{\eps,\delta}_{C_2}].
\ee
\end{lemma}
The method of proof will be to expand the products defining $\cI^\eps_C$
and then apply the GKS-inequalities \eqref{eq:gks} to the terms of the
expansion.  The subtlety is that the terms come with signs;  we show,
using Proposition \ref{prop:othercorrelations},
that all terms with negative sign actually vanish in expectation.
\begin{proof}
We will prove the first inequality, the proof of the second inequality
is almost identical.
Since $\prod_{i=1}^4\eps_i=+1$, we have that
$I^\eps_s+\bar I^\eps_s=\tfrac18(1+A_s+R_s)$
where $A_s=\prod_{i=1}^4 \s_i(s)$
and (with the argument $s$ suppressed for readability)
\be
R_s=\sum_{1\leq i<j\leq 4} \eps_i\eps_j \s_{i}\s_{j}.
\ee
Let $C$ be any of $C_1$, $C_2$ or $C_1\cup C_2$.
By expanding the product
over $s\in C$ we get that
\be
\cI^\eps_C=
\prod_{s\in C} \tfrac18(1+A_s+R_s)
=\frac1{8^{|C|}}\sum_{D\se C}\sum_{D_1\se D}
\prod_{s\in D_1} R_s \prod_{t\in D\sm D_1} A_t.
\ee
In this expression we expand the product over $s\in D_1$.
To write the expansion, we use the following notation.  We let 
$\Sigma(D_1)$ denote the set of sequences $(\{i(s),j(s)\})_{s\in D_1}$ 
of two-element subsets of $\{1,2,3,4\}$, indexed by $s\in D_1$.  
We may represent
such a subset $\{i(s),j(s)\}$ pictorially as an element of the set
$\{\lefthalfcup,-\hspace{-4pt}-,\righthalfcup,
\lefthalfcap,\mid\hspace{-1pt},\righthalfcap\}$, indicating the
orientations of the two edges selected.  Then,
\be\label{eq:I-expanded}
\cI^\eps_C=\frac1{8^{|C|}}
\sum_{\substack{D_1,D_2\se C\\D_1\cap D_2=\es}}
\sum_{\Sigma(D_1)}
\prod_{s\in D_1} \eps_{i(s)}\eps_{j(s)}\s_{i(s)}\s_{j(s)} 
\prod_{t\in D_2} A_t.
\ee
Now consider an arbitrary term
$T$ in the latter expansion, corresponding
to a choice of $D_1$, $D_2$ and sequence in $\Sigma(D_1)$.  We write
this term as 
\be\label{eq:T}
T=\prod_{s\in D_1} \eps_{i(s)}\eps_{j(s)}
 \prod_{s\in D_1}  \s_{i(s)}\s_{j(s)} 
\prod_{t\in D_2} A_t.
\ee
The following is the key claim:  if $T$ can be written as a
`product of stars', then it comes with a positive sign.  That
is, if there is a set $D_3\se V$ such that 
\be\label{eq:prod-of-stars}
\prod_{s\in D_1} 
\s_{i(s)}\s_{j(s)} 
\prod_{t\in D_2} A_t= 
\prod_{u\in D_3} A_u,
\ee
then 
\be\label{eq:pos-sign}
\prod_{s\in D_1} \eps_{i(s)}\eps_{j(s)} = +1.
\ee
Before proving the claim (i.e.\ that \eqref{eq:prod-of-stars}
implies \eqref{eq:pos-sign}) we show how to deduce the
result.  We may write 
\be
\cI^\eps_{C_1}  \cI^\eps_{C_2} 
=\frac{1}{8^{|C_1|+|C_2|} }\sum_{T_1,T_2} T_1T_2,
\ee
where $T_1$ and $T_2$ are terms of the form \eqref{eq:T}
for $C=C_1$ and for $C=C_2$, respectively, and 
$T=T_1T_2$ is of the form \eqref{eq:T}
for $C=C_1\cup C_2$.  We will use the GKS-inequality
\eqref{eq:gks} 
to see that all terms satisfy
$\EE_\bJ[ T_1T_2]\geq \EE_\bJ[ T_1]\EE_\bJ[ T_2]$.
By retracing the steps of the expansion,
this implies the result.
Now, if $T=T_1T_2$ does
\emph{not} satsify \eqref{eq:prod-of-stars}
then at least one of $T_1$ and $T_2$ does not satisfy it either.
By Propositions \ref{prop:ising} and \ref{prop:othercorrelations}, then 
$\EE_\bJ[ T_1T_2]=\EE_\bJ[ T_1]\EE_\bJ[ T_2]=0$.
If $T=T_1T_2$ \emph{does} satisfy \eqref{eq:prod-of-stars}
but one of $T_1$ and $T_2$ does not, then by 
Propositions \ref{prop:correlations} and \ref{prop:othercorrelations},
$\EE_\bJ[ T_1T_2]\geq0=\EE_\bJ[ T_1]\EE_\bJ[ T_2]$.
Finally, if all three of $T_1,T_2,T=T_1T_2$ satisfy
\eqref{eq:prod-of-stars} then the desired inequality
$\EE_\bJ[ T_1T_2]\geq\EE_\bJ[ T_1]\EE_\bJ[ T_2]$
follows from  the GKS-inequality \eqref{eq:gks}
and the fact that the terms have positive coefficients, thanks to the
claim.

We now prove the claim
that \eqref{eq:prod-of-stars}
implies \eqref{eq:pos-sign}.  If \eqref{eq:prod-of-stars} holds, then
since the $A_t$ satisfy $A_t^2=1$,  we get 
\be\label{eq:prod-of-stars-2}
\prod_{s\in D_1} 
\s_{i(s)}\s_{j(s)} 
\prod_{t\in D_2\sd D_3} A_t=1.
\ee
Recall that we represent each factor of the product over $s\in D_1$
as an element of the set of shapes
$\{\lefthalfcup,-\hspace{-4pt}-,\righthalfcup,
\lefthalfcap,\mid\hspace{-1pt},\righthalfcap\}$, 
indicating the two selected edges.
Similarly, each factor of the product over $t$ may be represented as a
$+$, indicating that all four edges are selected.
We think of the latter as consisting of two 
superimposed shapes, such as
$\lefthalfcup$ and $\righthalfcap$ (the precise choice does not
matter).  

We can assume that $D_1\cap (D_2\sd D_3)=\es$.  Then
we conclude
from \eqref{eq:prod-of-stars-2} that each edge $e\in E$
appears in the product either exactly twice, or not at all.  Consider
now the subset of edges which appear exactly twice (counting each such
edge once).  Due to the possible choices of `shapes' at each vertex,
this subset forms
an even subgraph (each vertex has even degree)
and thus decomposes as a collection of closed loops. 
We claim that
each such loop has the following property:
the number of vertices where it exhibits a shape from
$\{\lefthalfcup,\righthalfcap\}$ is even, similarly
the number of vertices where it exhibits a shape from
$\{\righthalfcup,\lefthalfcap\}$ is even,
and finally the number of vertices where it exhibits a shape from
$\{-\hspace{-4pt}-,\mid\hspace{-1pt}\}$ is even.  The latter claim can
be seen by induction, as follows. 
We may take the loops to be non-crossing.
For any non-crossing contractible
loop enclosing at least two squares,
we can decrease its enclosed
area one square by removing a corner.  As we do so, the number of
shapes from each of the four sets changes by an even amount.
Eventually the loop reduces to a single square, for which the claim
holds by inspection.  If the loop is not contractible, then a similar
reduction can be used to reduce it to a `straight' loop, which by our
assumption on $C_1\cup C_2$ has even length, meaning again that the
claim holds by inspection.

Finally, the sign on the left-hand-side of \eqref{eq:pos-sign}
can be written as 
\be \label{eq:pos-sign-2}
\prod_{s\in D_1} \eps_{i(s)}\eps_{j(s)} = 
\prod_{s\in D_1} \eps_{i(s)}\eps_{j(s)} 
\prod_{t\in D_2\sd D_3} \eps_1\eps_2\eps_3\eps_4.
\ee
Here we used that $\eps_1\eps_2\eps_3\eps_4=+1$.
For any choice of an even number of shapes from each of the sets 
$\{\lefthalfcup,\righthalfcap\}$,
$\{\righthalfcup,\lefthalfcap\}$ and
$\{-\hspace{-4pt}-,\mid\hspace{-1pt}\}$,
the product of the corresponding $\eps$'s
in \eqref{eq:pos-sign-2} is $=+1$, again using 
$\eps_1\eps_2\eps_3\eps_4=+1$.
The result follows. 
\end{proof}

\begin{proof}[Proof of Theorem \ref{thm:corr}]
For the toric code model, note that the operators $Q^\eps_C$ and $Q^{\eps,\delta}_C$ are
diagonal in the $\ket\om=\ket{\om}_\rz$-basis and satisfy 
$Q^\eps_C\ket\s=\cI^\eps_C(\s)\ket\s$ and $Q^{\eps,\delta}_C\ket\s=\cI^{\eps,\delta}_C(\s)\ket\s$.   
Then the result follows from
Proposition \ref{prop:ising} and Lemma \ref{lem:corr-ising}.

For the uniform 8-vertex model we deduce the result by fixing a
reference configuration $\rho\in\D_\eight$ and letting all
$J_s^\rz\to\oo$.  Since the limit of the Ising-measure
\eqref{eq:ising-prob} is uniform on $\s\in\Om_\eight$, the
distribution of $\s\cdot\rho$ is uniform on $\D_\eight.$
\end{proof}

\begin{remark}
It is possible, as in \cite{ardonne-etal}, 
to construct a quantum system
generalising the Kitaev model, whose
ground state  configurations 
correspond to eight-vertex configurations with the general weights
\be
\mu_{a,b,c,d}(\om)\propto
a^{\#\RN1+\#\RN2}
b^{\#\RN3+\#\RN4}
c^{\#\RN5+\#\RN6}
d^{\#\RN7+\#\RN8}
\ee
depending on the numbers 
$\#\RN1,\dotsc,\#\RN8$ of vertices of the different types. 
However, the operators that must be added
to the hamiltonian 
to achieve this do not satisfy the required non-negativity.
This suggests that the inequalities in Theorem \ref{thm:corr} may
not hold for non-uniform weights. 
\end{remark}

Using almost the same argument as for Lemma \ref{lem:corr-ising} we
can also prove the following:
\begin{proposition}
For any $C,D\se V$ and any $\eps\in\{-1,+1\}^4$ satisfying
$\prod_{i=1}^4 \eps_i=+1$, as above, 
\be
\langle Q^\eps_ C \textstyle \prod_{u\in D} Z_u\rangle
\geq \langle Q^\eps_ C \rangle
\langle \textstyle \prod_{u\in D} Z_u\rangle.
\ee
In particular, for any $s\in V$,
\be\label{eq:deriv}
\tfrac{\partial}{\partial J_s}
\langle Q^\eps_ C \rangle=
\langle Q^\eps_ C Z_s\rangle
-\langle Q^\eps_ C \rangle \langle Z_s\rangle\geq 0.
\ee
\end{proposition}
\begin{proof}
We have
\be
\langle Q^\eps_C \textstyle \prod_{u\in D} Z_u\rangle
=
\EE_\bJ[ \cI^\eps_C(\s) \textstyle \prod_{u\in D} A_u(\s)],
\ee
where $A_u(\s)=\prod_{i=1}^4 \s_i(s)$ as before.  The same working as
for \eqref{eq:I-expanded} combined with the fact that 
$A_u(\s)=+1$ gives
\be
\cI^\eps_C \prod_{u\in D} A_u(\s)
=\frac1{8^{|C|}}
\sum_{\substack{D_1,D_2\se C\\D_1\cap D_2=\es}}
\sum_{\Sigma(D_1)}
\prod_{s\in D_1} \eps_{i(s)}\eps_{j(s)}\s_{i(s)}\s_{j(s)} 
\prod_{t\in D_2\sd D} A_t.
\ee
From there the same argument as in Lemma \ref{lem:corr-ising}
applies.  
\end{proof}

A standard consequence of \eqref{eq:deriv}
is the existence of infinite-volume limits of correlation functions of
the form $\langle Q^\eps_C \rangle$;  see e.g.\ 
\cite[Proposition 4]{BLU} for arguments of this nature.

\section{Further results for the uniform 8-vertex model}
\label{sec:various}

In this section we discuss some further properties of the uniform
eight-vertex model that are natural to consider
from the viewpoint of `plaquette-flipping'.  

\subsection{Communicating classes}
\label{ssec:comm}

We start by considering the communicating classes of the
dynamics \eqref{eq:rw-d} in the case when all 
$J_p>0$.   Recall that $\Lambda$ is an
$m\times n$ torus.

\begin{proposition}\label{prop:CCtorus} 
We have that
$|\Delta_{\eight}|=2^{|V|+1}$, and there are four
communicating classes for the dynamics \eqref{eq:rw-d},
each of size $2^{|V|-1}$.   We may move between them
by reversing arrows along a non-contractible path in $\Lambda$.
\end{proposition}

\begin{proof} 
We start by showing that
$|\Delta_{\eight}|=2^{|V|+1}$. 
First, clearly $|\Delta|=2^{|E|}=2^{2|V|}$.  
By fixing a reference-configuration $\rho$
as in \eqref{eq:rw-d} we can encode each element of $\D$
using an element $v\in (\FF_2)^E$ where $\FF_2=\{0,1\}$
is the two-element field.  Then, for each $s\in V$,
the constraint that $s$ 
has an even number of incoming arrows becomes a linear constraint over
$\FF_2$, namely $g^s\cdot v=0$ where (similarly to \eqref{eq:gp}) 
\be
g^s_e=\left\{
\begin{array}{ll}
1, & \mbox{if } e\sim s,\\
0, & \mbox{otherwise},
\end{array}
\right.
\ee
and $\cdot$ is the scalar product.
These constraints are \emph{not} linearly independent, 
since $(\sum_{s\in V} g^s)_e=0$ for each $e\in E$.  However, this is
the only linear relation:  any other linear relation would have to be
of the form $\sum_{s\in A} g^s\equiv 0$ for some proper subset $A$ of
$V$.  But then there must be some
edge $e\in E$ with precisely one end point in $A$ and then
$(\sum_{s\in A} g^s)_e=1\neq 0$.
From this and  the rank-nullity theorem we get that
$|\Delta_{\eight}|=2^{2|V|-(|V|-1)}=2^{|V|+1}$, as claimed.

Now it is simple to see that there are four communicating classes for
the dynamics. There are $|V|$ plaquettes in $F$ and hence $2^{|V|}$
possible sums $\sum_{p\in A}g^{p}$ for $A\se V$. However, we have
that $\sum_{p\in A}g^p= \sum_{p\in V\setminus A}g^p$ (since
$\sum_{p\in V}g^p\equiv 0$).  Hence,
starting from any reference configuration $\rho\in\D_\eight$,
there are $2^{|V|-1}$ configurations reachable by flipping plaquettes.
 This shows that there are four communicating classes.

Lastly, let $P \subset E$ be a non-contractible path of edges in
$\Lambda$, it is clear that such a path is not the boundary of a
collection of plaquettes $A\subset F$. This means that the
configuration obtained by reversing arrows on $P$ is not reachable
from the reference configuration via the dynamics. There are two
homotopy classes of non-contractible paths on the torus and reversing
arrows along a path of edges in one or two of the classes moves us
between the four communicating classes.
\end{proof}

\begin{remark}
\phantom{hej}
\begin{enumerate}[leftmargin=*]
\item
Because our dynamics for the uniform
eight-vertex model correspond to a random
walk on the hypercube $(\FF_2^+)^F$, we can import results
about mixing times from the literature.
For example, the mixing time
for the dynamics is $\log(|V|)$ (recall that the Poisson process of
steps has rate $|F|$ rather than the usual rate 1).
The interested
reader can consult \cite{levin} and references therein for detailed
statements.
\item The limiting distribution of  random
walk on the hypercube $(\FF_2^+)^F$ is uniform.  This leads to a
method for sampling from the uniform eight-vertex distribution
$\mu(\cdot)$:  first sample uniformly a representative 
$\rho\in\D_\eight$ from each of the four communicating classes to act
as  reference-configuration, then toss independent coins for each of
the plaquettes for whether to `flip' the plaquette or not.
\item  If $s_1,s_2\in V$ are vertices not adjacent to any common
  plaquette $p$, then by the previous item, the vertex-types at $s_1$
  and $s_2$  can be written as functions of independent random
  variables (states of the plaquettes surrounding them).
Thus their states are independent.
\end{enumerate}
\end{remark}

\subsection{Emptiness formation probability}
\label{ssec:efp}

In this subsection we allow to view $\L$ 
not only as a torus, but also 
as a subset of $\ZZ^2$ with a boundary.  
In that case we replace each of the edges of the form 
$\{(m,y),(1,y)\}$ with two boundary-edges
$\{(m,y),(m+1,y)\}$ and $\{(0,y),(1,y)\}$,
and similarly replace each 
$\{(x,n),(x,1)\}$ with two boundary-edges
$\{(x,n),(x,n+1)\}$ and $\{(x,0),(x,1)\}$.

First we consider the \emph{domain wall boundary
condition}.  This means that the top and bottom boundary 
edges (of the form $\{(0,y),(1,y)\}$ or $\{(m,y),(m+1,y)\}$) 
receive the fixed orientation pointing \emph{in} towards
$\Lambda$, and that the left and right boundary-edges (of the 
form $\{(x,0),(x,1)\}$ or $\{(x,m),(x,m+1)\}$) are fixed to point
\emph{out} of $\Lambda$.  
See Figure \ref{fig:DWminus1} for an illustration. 

We emphasise $m$ and $n$ by
using the notation $\Lambda_{m,n}$.
We denote
by $\Delta_{\eight}^{\dw}(\Lambda_{m,n})$ the set of eight-vertex
configurations on $\Lambda_{m,n}$ with domain wall boundary conditions
as described above.


\begin{proposition}\label{prop:DWpartition} 
We have that
\be
|\Delta_{\eight}^{\dw}(\Lambda_{m,n})|=
\left\{\begin{array}{ll}
0, & \mbox{if $m+n$ is odd,}\\
2^{(m-1)(n-1)}, & \mbox{if $m+n$ is even}.
\end{array}\right.
\ee
When non-empty,  $\Delta_{\eight}^{\dw}(\Lambda_{m,n})$
consists of a single communicating class for the 
dynamics \eqref{eq:rw-d}. 
\end{proposition}
\begin{proof} 
We begin by showing that $\Delta_{\eight}^{\dw}(\Lambda_{m,n})\neq\es$
if and only if $m+n$ is even.  This is done in three steps.

Step 1:  if $\Delta_{\eight}^{\dw}(\Lambda_{m-1,n-1})\neq\es$
then $\Delta_{\eight}^{\dw}(\Lambda_{m,n})\neq\es$.  Indeed,
Figure \ref{fig:DWminus1} illustrates how an element of 
$\Delta_{\eight}^{\dw}(\Lambda_{m-1,n-1})$ can be `extended'
to an element of $\Delta_{\eight}^{\dw}(\Lambda_{m,n})$.

\begin{figure}
  \begin{center}
  \includegraphics[scale=.25]
{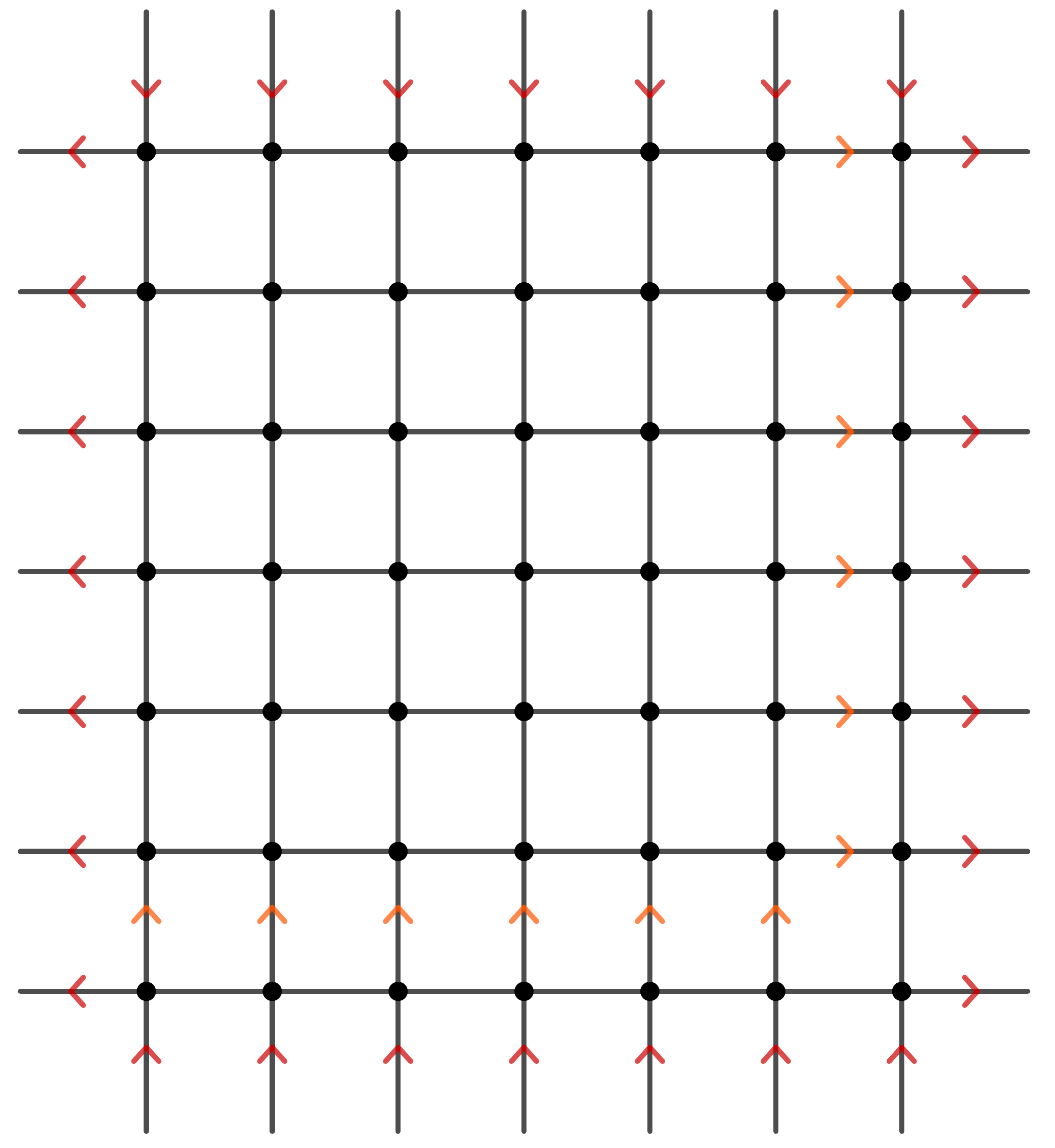} \qquad
    \includegraphics[scale=.25] 
{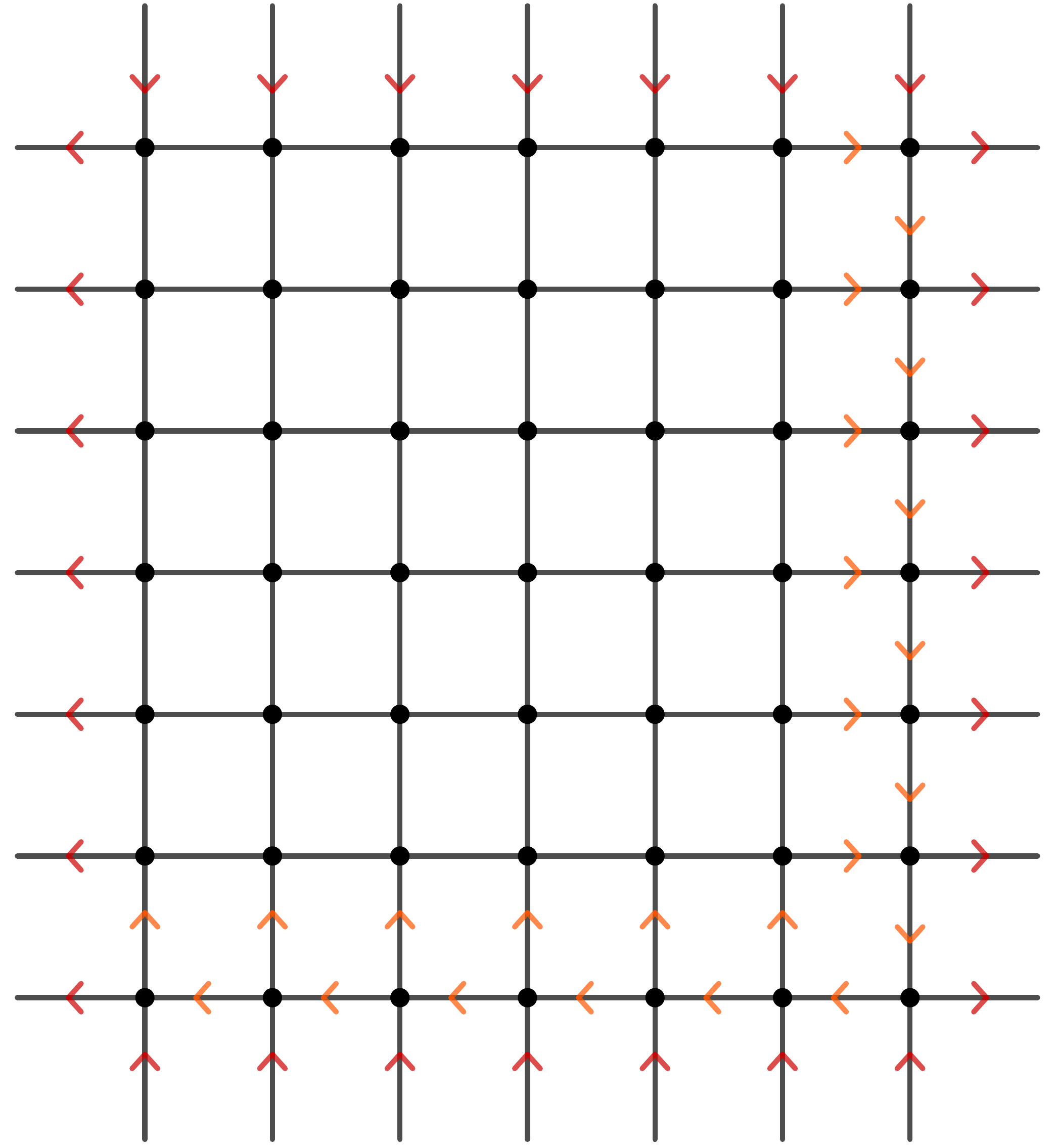}  
 \end{center}
 \caption{Left: extending $\L_{m-1,n-1}$ to 
$\L_{m,n}$ with both having domain wall boundary conditions.
Right: 
$\Delta_{\eight}^{\dw}(\Lambda_{m-1,n-1})\neq\es$
implies $\Delta_{\eight}^{\dw}(\Lambda_{m,n})\neq\es$.
} 
\label{fig:DWminus1}
\end{figure}

Step 2:  conversely, if $\Delta_{\eight}^{\dw}(\Lambda_{m,n})\neq\es$
then $\Delta_{\eight}^{\dw}(\Lambda_{m-1,n-1})\neq\es$.
Indeed, fix an element $\d\in\Delta_{\eight}^{\dw}(\Lambda_{m,n})$
and consider the bottom row of vertical edges and the rightmost column
of horizontal edges in $\L_{m,n}$ (see Figure \ref{fig:DWBCnoodd}).
If all those vertical edges are oriented \emph{in}, and all the
horizontal ones are oriented \emph{out}, then the restriction of $\d$
to $\L_{m-1,n-1}$ is an element of 
$\Delta_{\eight}^{\dw}(\Lambda_{m-1,n-1})$.  Otherwise, 
the edges pointing the `wrong way' (i.e.\ out for vertical edges on
the bottom, in for horizontal edges on the right side) will be called
\emph{faults}.  It is not hard to check that the number of faults must
be even.  Then, pair up the successive
faults and mark the plaquettes between the
pairs, as in Figure \ref{fig:DWBCnoodd}.  Flipping all the marked
plaquettes maps $\d$ to a configuration whose restriction to 
$\L_{m-1,n-1}$ is an element of 
$\Delta_{\eight}^{\dw}(\Lambda_{m-1,n-1})$.

\begin{figure}
  \includegraphics[width=1\textwidth,
height=0.5\textwidth]{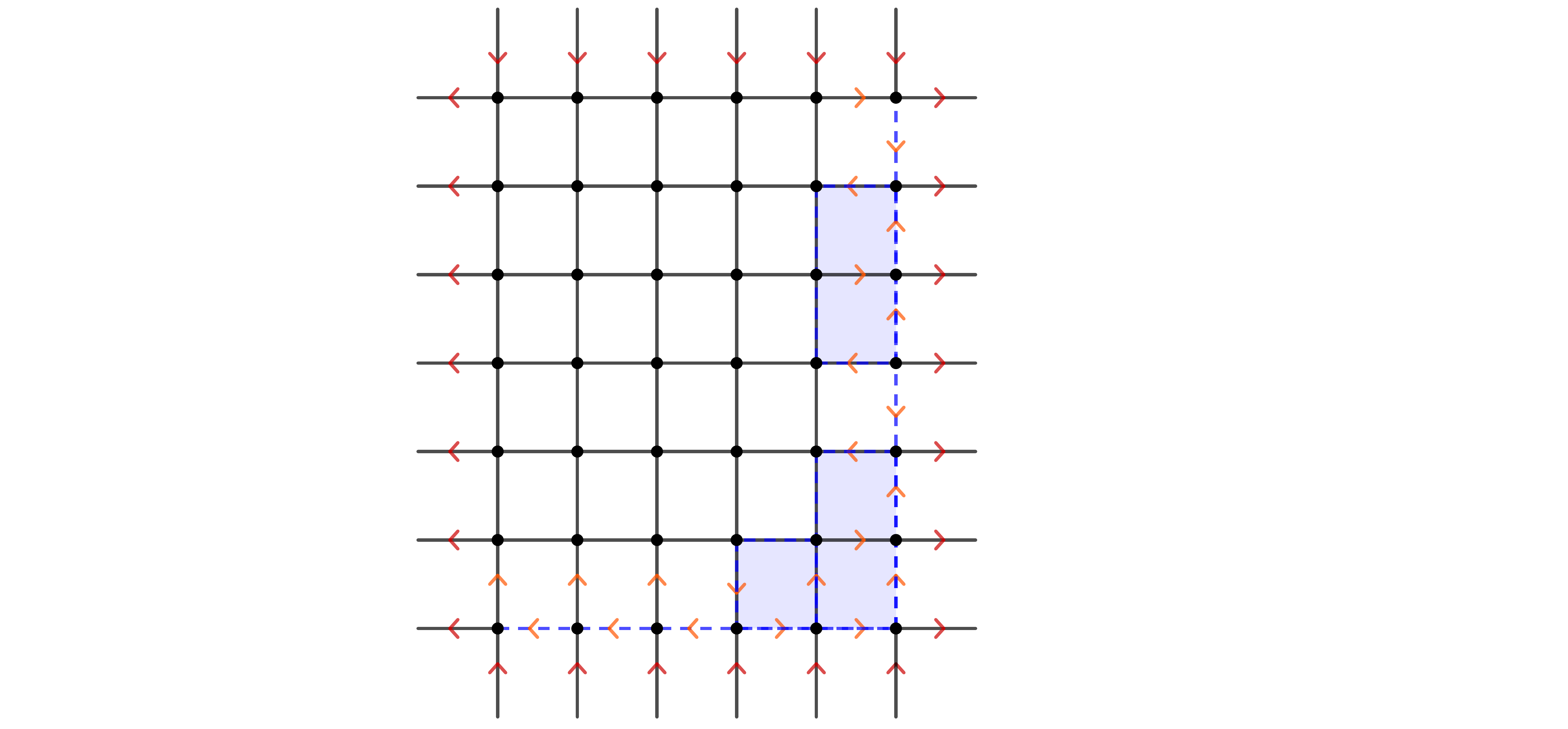} \hspace{-20pt} 
 \caption{An arbitrary element of
$\Delta_{\eight}^{\dw}(\Lambda_{m,n})$ can be mapped to one whose
restriction to $\Lambda_{m-1,n-1}$ also has domain wall boundary
condition.  The highlighted plaquettes separate pairs of
\emph{faults}, i.e.\ arrows pointing the opposite way to the boundary
condition.  Flipping the highlighted plaquettes gives the desired
element of $\Delta_{\eight}^{\dw}(\Lambda_{m-1,n-1})$. 
}
\label{fig:DWBCnoodd}
\end{figure}

Step 3:  from the previous two steps we see that 
$\Delta_{\eight}^{\dw}(\Lambda_{m,n})\neq\es$
if and only if $\Delta_{\eight}^{\dw}(\Lambda_{m-1,n-1})\neq\es$.  If
$m\geq n$ this holds if and only if 
$\Delta_{\eight}^{\dw}(\Lambda_{m-n+1,1})\neq\es$;  if
$n\geq m$ it holds if and only if 
$\Delta_{\eight}^{\dw}(\Lambda_{1,n-m+1})\neq\es$.
It is straightforward to check that the latter sets are non-empty if
and only if $m-n$ is even, which is equivalent to $m+n$ being even. 

To determine the number of configurations when $m+n$ is even,
 we use a simple counting argument. 
First, let us count the number of choices at each vertex if we start
at the top left corner $(1,n)$ and proceed
left to right and then top to bottom.
As we proceed, each time we `arrive' at a new vertex \emph{except} the
rightmost column or bottom row,
exactly two incident arrows are already fixed, leaving us with exactly
two choices.  In the rightmost column and bottom row, 
\emph{three} incident
arrows are fixed when we arrive, leaving us at most one (possibly
no) choice per vertex.
This gives that
\be\label{eq:DW-ub}
|\Delta_{\eight}^{\dw}(\Lambda_{m,n})|\leq 2^{(m-1)(n-1)}.
\ee
On the other hand, if 
$\Delta_{\eight}^{\dw}(\Lambda_{m,n})\neq \emptyset$, 
fix some $\d\in \Delta_{\eight}^{\dw}(\Lambda_{m,n})$.
An argument similar to that of Proposition \ref{prop:CCtorus}
shows that, in this case, the number of configurations reachable by
flipping plaquettes equals the total number of configurations.  Since
there are $(m-1)(n-1)$ plaquettes, this gives that
\be\label{eq:DW-lb}
|\Delta_{\eight}^{\dw}(\Lambda_{m,n})|\geq 2^{(m-1)(n-1)}.
\ee
Combining \eqref{eq:DW-ub} and \eqref{eq:DW-lb}
shows that, when non-empty, 
$\Delta_{\eight}^{\dw}(\Lambda_{m,n})$ has size 
$2^{(m-1)(n-1)}$.
The claim about irreducibility follows from the proof of
\eqref{eq:DW-lb}.
\end{proof}

We now turn our attention to the so-called \emph{emptiness formation
probability} of the model. This is the probability that a 
fixed column of
horizontal edges has all of its arrows pointing to the left (say).
 In this case we say that the column is \emph{empty}. 
Denote by $\EF_r$
the event that the column of horizontal edges whose left end-points
have first coordinate $r$ is empty, see Figure \ref{fig:DWBCEFP}.
In what follows, 
$\mu^\dw(\cdot)$ denotes the uniform probability measure on
$\Delta_{\eight}^{\dw}(\Lambda_{m,n})$.
We calculate the probability of $\EF_r$
in the case of the torus and the case of domain wall boundary
conditions.  We begin with the domain wall case.

\begin{figure}
 \begin{center} 
  \includegraphics[scale=.25]{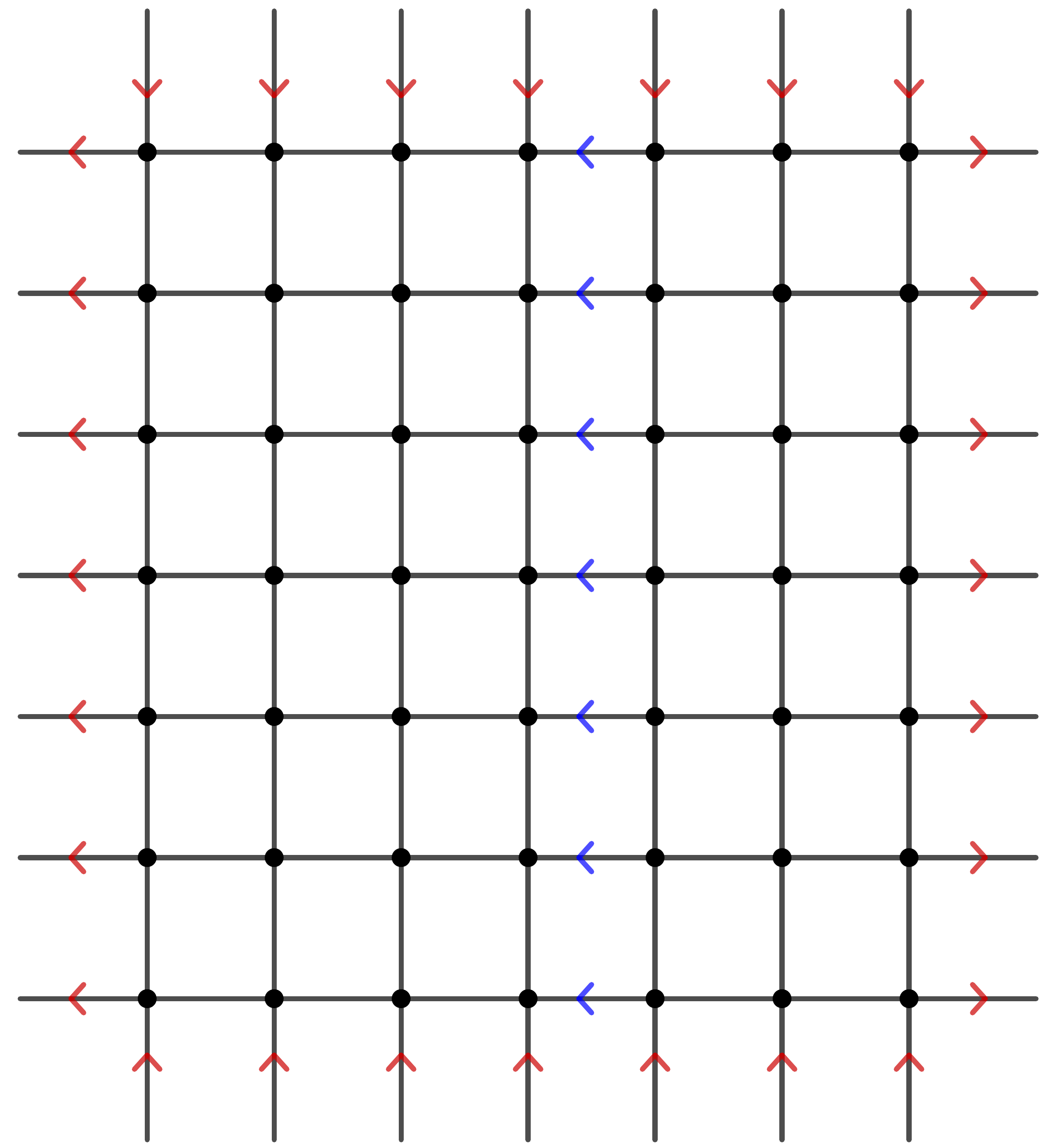} 
  \end{center}
 \caption{Example of an empty column in a box with domain wall
boundary condition. Here $m=n=7$ and $r=4$.}
\label{fig:DWBCEFP}
\end{figure}

\begin{proposition}\label{prop:DWBCEFP} 
Assume that $m+n$ is even so that 
$\Delta_{\eight}^{\dw}(\Lambda_{m,n})\neq\es$.
Then
$\mu^{\dw}(\EF_r)=2^{-(m-1)}\1_{\{r\in 2\NN\}}.$
\end{proposition}
\begin{proof} 
The proof is similar to that of
Propositon \ref{prop:DWpartition}.  First we show that
$\EF_r\neq \emptyset$ if and only if $r$ is even. 
Indeed, by Propositon \ref{prop:DWpartition}, the part of $\L_{m,n}$
to the right of the empty column can be `filled in' if and only if
$n-r+m$ is even.  Since $m+n$ is even, this is equivalent to $r$ being
even.  
To fill in the part to the left of the empty column, 
we can make the top row alternate between 
types $\RN6$ and $\RN7$ (also this is possible if and only if
$r$ is even) and the rest all type $\RN4$. 
(Recall Figure \ref{fig:vertices} for the vertex types.)
This shows that
$\EF_r\neq \emptyset$ if and only if $r$ is even. 

Assuming then that $r$ is even, let us count the number of choices at
each vertex in the two parts, going
from left to right and then top to bottom in each.
We see that
$(r-1)(m-1)+(n-r-1)(m-1)=(m-1)(n-2)$ vertices have two choices of
vertex type and the remaining vertices have at most one choice.  From
this and Proposition \ref{prop:DWpartition} we see that
\be
\mu^{\dw}(\EF_r)\leq 2^{(m-1)(n-2)}\big/
2^{(m-1)(n-1)}=2^{-(m-1)}.
\ee
On the other hand,  consider the $(m-1)(n-1)-(m-1)$ plaquettes
that may be reversed without affecting the empty column.
Similarly to the proof of Proposition \ref{prop:DWpartition} 
we see that
each distinct choice gives a different configuration. 
From this we have that
\be
\mu^{\dw}(\EF_r)\geq 2^{(m-1)(n-1)-(m-1)}\big/
2^{(m-1)(n-1)}=2^{-(m-1)},
\ee
which completes the proof.
\end{proof}

In the case when $\L_{m,n}$ is viewed as a torus,
the emptiness formation probability is independent of the position
of the empty column.  We hence simply denote by $\EF$ the event
that an arbitrary, fixed, column of the torus is empty.  Recall that
$\mu(\cdot)$ is the uniform distribution on eight-vertex
configurations. 

\begin{proposition} 
Consider $\Lambda_{m,n}$ with sides
identified to form a torus.  Then
$\mu(\EF)=2^{-m}.$
\end{proposition}
\begin{proof} 
Let us fix the column that is to be empty as the rightmost column. 
First, it is simple to see that $\EF\neq\es$ as the
configuration consisting of all vertices being type $\RN2$ has all columns
empty.  For an upper bound on $\mu(\EF)$ we again use a simple counting
argument.  By fixing the
vertex types starting from left to right and then top to bottom, we
see that the leftmost column has one fixed incident arrow,
hence these $n-1$ vertices have four choices.
The `bulk' $(m-2)(n-1)+1$ vertices have two
choices and the remaining vertices have at most one choice. We
therefore have that
\be
\mu(\EF)\leq 2^{2(n-1)+(m-2)(n-1)+1}\big/2^{mn+1}=2^{-m}.
\ee
On the other hand, consider the $m(n-1)$ plaquettes whose
bounding arrows can be reversed without affecting the empty column.
Note that distinct choices of which plaquettes to reverse give
different configurations (we can now only reverse plaquettes in one of
$A,A^c\subset F$ and not both because one of $A, A^c$ contains the
plaquettes of the empty column).  Thus we can obtain $2^{m(n-1)}$ distinct
configuration in $\EF$ from a fixed reference
configuration.  Next, note that if we reverse all arrows along a
straight, vertical, non-contractible path of vertices away from the
rightmost column, then we obtain a
configuration in $\EF$ that can not be reached from the reference configuration
by reversing arrows around plaquettes. By now reversing around plaquettes
from this new eight-vertex configuration we find another $2^{m(n-1)}$
distinct configuration with the empty column. This gives that
\be
\mu(\EF)\geq 2\cdot 2^{m(n-1)}\big/2^{mn+1}=2^{-m},
\ee
which matches our upper bound.
\end{proof}

\subsection{Entropy}

Now we consider the \emph{entropy} of the model, 
meaning that we compare the
number of eight-vertex configurations for various boundary
conditions.  We will suppose that $m,n$ are
 fixed such that $m+n$ is even and we revert to the notation $\L$
 without subindices.
Let $\eta$ be a fixed assignment of arrows to the boundary
edges (i.e.\ the edges $\{(m,y),(m+1,y)\}$, $\{(0,y),(1,y)\}$,
$\{(x,n),(x,n+1)\}$ and $\{(x,0),(x,1)\}$)
and denote by $\Delta^{\eta}_{\eight}(\L)$ the set of
eight-vertex configurations on $\Lambda$ with the boundary
configuration $\eta$.
We call $\eta$ \emph{valid} if
$\Delta^{\eta}_{\eight}(\L)\neq\emptyset$.

\begin{proposition}\label{prop:possiblebc} 
A boundary condition $\eta$ is valid if and only
if the number of arrows around the boundary of $\Lambda$ that point
into $\Lambda$ is even.   In this case
$|\Delta^{\eta}_{\eight}(\L)|=2^{(m-1)(n-1)}$.
\end{proposition}
\begin{proof} 
If $\Delta^{\eta}_{\eight}\neq \emptyset$, then
the same argument as in the proof of
Proposition \ref{prop:DWpartition} shows that
all configurations in $\Delta^{\eta}_{\eight}$ are reachable by
flipping plaquettes, and thus
$|\Delta^{\eta}_{\eight}|=2^{(m-1)(n-1)}$.  It remains to determine
when $\Delta^{\eta}_{\eight}\neq \emptyset$.

Let $\eta_0$ be the boundary condition such that all
arrows are $\rightarrow$ or $\uparrow$.
By taking the 
configuration on $\Lambda$ with every vertex of type $\RN{1}$ we see
that $\Delta^{\eta_0}_{\eight}\neq\es$.
This boundary condition does
indeed have an even number ($m+n$) of inward arrows. Now, 
a similar argument as for 
Proposition \ref{prop:CCtorus} shows that any valid boundary condition 
$\eta$ is obtainable by flipping the exterior boundary plaquettes.
Then $\eta$ differs from $\eta_0$
at an even number of boundary edges, and hence still has an
even number of inward pointing arrows. Indeed, if $\eta$ differs from
the reference boundary condition in an even number of places then we
can pair off differing edges into neighbouring pairs and reverse arrow
around plaquettes between the pairs, similarly to Figure
\ref{fig:DWBCnoodd}. On the other hand, if $\eta$ differs from 
$\eta_0$ at an odd number of edges then there is
no set of plaquettes that can be flipped to obtain $\eta$.
\end{proof}

The next result can be interpreted as saying that 
the entropy of the uniform
eight-vertex model on $\Lambda$ is realised (up to a factor that
is exponential only in the size of the boundary) for any fixed valid
boundary condition.  It is an immediate consequence of Propositions
\ref{prop:CCtorus} and \ref{prop:possiblebc}.

\begin{proposition}\label{prop:entropy} 
Let $\eta$ be a valid boundary condition on
$\Lambda$.   Then
\[
\frac{|\Delta_{\eight}|}{|\Delta^{\eta}_{\eight}|}=2^{m+n-1}=2^{O(\partial
\Lambda)}.
\]
\end{proposition}

\appendix

\section{The toric code}\label{app:codes}

In this appendix we summarise some of the original motivation
for the toric code model, starting with basic properties of quantum
codes.   

Recall
that classical codes store information in sequences of numbers 0 or 1,
each of which is called a bit.  For quantum codes, the bits 0 or 1 are
replaced by so-called \emph{qubits} which are more elaborate objects.
A qubit may be defined as \emph{an irreducible two-dimensional
  representation of $\mathfrak{su}_2(\CC)$} and a quantum code of
length $n$ as \emph{an $n$-fold tensor product of qubits}.   
Let us now unpack these definitions.

Consider as before the two-dimensional 
vector space $\CC^2$ and with
standard basis   
$\ket{+}=\big(\begin{smallmatrix}1\\0\end{smallmatrix}\big)$
and 
$\ket{-}=\big(\begin{smallmatrix}0\\1\end{smallmatrix}\big)$
(these are often denoted $\ket{1}$ and $\ket{0}$ instead,
but $\ket{+}$  and $\ket{-}$ is a more natural choice for us).
The Pauli matrices \eqref{eq:pauli}, under addition and
multiplication, generate an 
algebra of two-by-two 
Hermitian matrices with trace 0.  Together with the
underlying vector space $\CC^2$ on which the matrices act, this algebra
of matrices is a \emph{representation} of (the universal
enveloping algebra of) $\mathfrak{su}_2(\CC)$.  That is, they form a
single qubit by our definition.  

It is a crucial point that there are other ways to represent a qubit
than the specific construction above.  That is, one may find other
two-dimensional vector spaces $\VV$ (over $\CC$), 
together with matrices $X,Y,Z$
acting on $\VV$, that have ``all relevant properties'' of $\CC^2$
together with the $\sone,\stwo,\sthree$.
More precisely, there are other representations of
$\mathfrak{su}_2(\CC)$ which are \emph{isomorphic} to the one above
and therefore also form a single qubit.  
To check that $\VV$ together with $X,Y,Z$ form a qubit, one needs to
check that $X,Y,Z$ 
satisfy the following 
relations:
\be\label{eq:commrel}
[X,Z]:=XZ-ZX=-2iY=-2ZX,
\qquad \mbox{and} \qquad  X^2+Y^2+Z^2=3\one.
\ee

For classical codes, one obtains protection from errors by using
\emph{redundance};  that is, a single bit is encoded in a sequence of
several bits of length $n>1$, where the extra bits are used to detect
possible errors of transmission.  For quantum codes, the analogous
setting is obtained using the $n$-fold tensor product 
$(\CC^2)^{\otimes n}$. 
For each $k\in\{1,\dotsc,n\}$ one then has a
\emph{physical qubit} consisting of copies 
$\sone_k,\stwo_k,\sthree_k$ of the Pauli matrices which act only
on the $k$ entry.
To obtain a quantum code with desirable properties, one sets  this up
in such a way that $(\CC^2)^{\otimes n}$ has at least one 
two-dimensional 
subspace carrying a representation of 
$\mathfrak{su}_2(\CC)$, that is, forms a qubit
(other than the ones obtained using the factors
$\sone_k,\stwo_k,\sthree_k$).  Such a 
sub-representation is then
called a \emph{logical qubit}.

For the toric code, we use the setting described above with
$n=|E|$, that is the physical qubits are indexed by the edge-set $E$
of the torus.  
To identify logical qubits we use the operators
\be\label{eq:AB}
A_s=\prod_{e\sim s} \sthree_e,\qquad
B_p=\prod_{e\sim p} \sone_e,
\qquad s\in V,\quad p\in F.
\ee
(These are the same as $X_s$ and $Z_p$ in \eqref{eq:ZX} but here we
use the more common notation \eqref{eq:AB}.)
The relevant subspace of $(\CC^2)^{\otimes E}$ is
\be\label{eq:kitaev-L}
\cL=\big\{\ket{\xi}\in (\CC^2)^{\otimes E}:
A_s\ket{\xi}=B_p\ket{\xi}=\ket{\xi},
\mbox{ for all } s\in V, p\in F
\big\},
\ee
i.e.\ the subspace stabilised by all $A_s$ and $B_p$.  
We will see that $\cL$ 
carries  \emph{two} logical qubits.

Let us look more closely at the condition $A_s\ket{\xi}=\ket{\xi}$.  
Note that
\be
A_s |\xi\rangle = \big(\prod_{e\sim s} \xi_e\big)|\xi\rangle.
\ee
Thus, the condition $A_s\ket{\xi}=\ket{\xi}$ is identical to the
constraint defining $\Om_\eight\se\Om$ as the set
of configurations
$\s$ satisfying $\prod_{e\sim s}\s_1=+1$ for all $s\in V$,
In particular, we may
identify $\cL$ with a subset of $\Om_\eight$.

Translating the multiplicative constraints 
defining $\cL$ into linear constraints, one may use the 
same reasoning (rank-nullity) as for
Proposition  \ref{prop:CCtorus}
to conclude that $\mathrm{dim}(\cL)=4$.  Thus $\cL$ has the correct
dimension for carrying two qubits.  To see that it indeed does, we need
to define operators $X_1,Z_1$ and $X_2,Z_2$ on $\cL$ which commute
(for differing indices) and satisfy \eqref{eq:commrel}
(for matching indices).  Let $L_1\se E$ be the set of edges of the
form $\{(x,1),(x+1,1)\}$ for $1\leq x\leq m$.  Thus $L_1$ is a
horizontal path that wraps around the torus.  Also let $L_1'$ be the
set of edges of the form $\{(1,y),(2,y)\}$ for $1\leq y\leq n$ forming
a vertical `ladder' around the torus.   Define
\be
X_1=\prod_{e\in L_1} \sone_e,\qquad
Z_1=\prod_{e\in L_1'} \sthree_e.
\ee
Similarly let $L_2$ be the set of edges of the
form $\{(1,y),(1,y+1)\}$ for $1\leq y\leq n$
and let $L_2'$ be the set of edges of the form 
$\{(x,1),(x,2)\}$ for $1\leq x\leq m$, respectively forming
a vertical path and a horizontal `ladder', and define 
\be
X_2=\prod_{e\in L_2} \sone_e,\qquad
Z_2=\prod_{e\in L_2'} \sthree_e.
\ee
One may check that these operators indeed have the desired
properties, the key observation being that the paths share either zero
or exactly one edge.

Let us briefly describe the error-correction properties of the code
$\cL$ on an intuitive level.  
Imagine that a message is encoded as an element of the space
$\cL$, but an error occurs in transmission so that the received
message is no longer an element of $\cL$.  The most
likely cause for this is that a \emph{minimal} (non-zero) number of
the constraints $A_s\ket\xi=\ket\xi$ and $B_p\ket\xi=\ket\xi$ 
is broken.  By parity constraints, this minimal number is two
(since $\prod_{s\in V}A_s=1$ and $\prod_{p\in F}B_p=1$).
Suppose that $A_{s_1}\ket\xi\neq \ket\xi$
and $A_{s_2}\ket\xi\neq \ket\xi$.  In terms of arrow configurations
$\om\in\Om$,  this means that $s_1$ and $s_2$ are \emph{faults}, i.e.\
the total number of in- (or out-) pointing arrows is odd.  Then,
flipping all the qubits along a path between $s_1$ and $s_2$ gives an
element of $\cL$ again.  It is plausible that the element we obtain in
this way is the message that was sent.


\begin{thebibliography}{99}

\bibitem{AR}
D. Allison and N. Reshetikhin,
{\em Numerical study of the 6-vertex model with domain wall boundary
  conditions},
Ann. Inst. Fourier, 55(6), 2005

\bibitem{AFH}
R. Alicki, M. Fannes, and M. Horodecki,
{\em A statistical mechanics view on {K}itaev's
proposal for quantum memories}
J. Phys. A: Math. Theor. 40 6451 

\bibitem{ardonne-etal}
E. Ardonne, P. Fendley, and E. Fradkin,
{\em Topological order and conformal quantum critical points},
Ann. Phys. 310(2) pp. 493-551, 2004

\bibitem{Bax}
R. J. Baxter,
{\em Eight-Vertex Model in Lattice Statistics},
Phys. Rev. Lett. vol. 26 pp. 832-833, 1971

\bibitem{Baxbook}
R. J. Baxter,
{\em Exactly solved models in statistical mechanics},
London Academic Press Limited, 1982

\bibitem{BLU}
C. Benassi, B. Lees, and D. Ueltschi,
{\em Correlation Inequalities for the Quantum XY Model},
J.  Stat. Phys. 164 1157-1166, 2016

\bibitem{DKMO}
H. Duminil-Copin, A. Karrila, I. Manolescu, and M. Oulamara,
{\em Delocalization of the height function of the six-vertex model}
{arXiv2012.13750}, 2020

\bibitem{DKMO2}
H. Duminil-Copin,  K. K. Kozlowski, D. Krachun, I. Manolescu, 
and Tikhonovskaia,
{\em On the six-vertex model's free energy}. arXiv preprint
arXiv:2012.11675, 2020.

\bibitem{FandV}
S. Friedli and Y. Velenik,
{\em Statistical Mechanics of Lattice Systems: A Concrete Mathematical Introduction},
Cambridge University Press, 2017

\bibitem{FMMC}
A. Fowler, M. Mariantoni, J. M. Martinis, and A. N. Cleland,
{\em Surface codes: Towards practical large-scale quantum computation}
Physical Review A, 86(3), p.032324, 2012

\bibitem{FW}
C. Fan and F- Y- Wu,
{\em General Lattice Model of Phase Transitions}
Phys. Rev. B 2, 723, 1970

\bibitem{GP}
A. Glazman and R. Peled,
{\em On the transition between the disordered and antiferroelectric
  phases of the 6-vertex model}. arXiv preprint arXiv:1909.03436 

\bibitem{Grover1}
L. K. Grover,
{\em Quantum Mechanics Helps in Searching for a Needle in a Haystack},
Phys. Rev. Lett. Vol. 79 325, 1997

\bibitem{Grover2}
L. K. Grover,
{\em A Fast Quantum Mechanical Algorithm for Database Search},
Ann. ACM Symp. Th. Comp. pp. 212-219, 1996


\bibitem{kitaev}
A. Yu. Kitaev,
{\em Fault-tolerant quantum computation by anyons},
Ann. Phys. 303(1) pp. 2-30, 2003


\bibitem{levin}
D. A. Levin, Y. Peres, and E. L. Wilmer,
{\em Markov Chains and Mixing Times}
AMS, 2009

\bibitem{Lieb671}
E. H. Lieb,
{\em Residual Entropy of Square Ice}
Phys. Rev. vol. 162 pp. 162-172, 1967

\bibitem{Lieb672}
E. H. Lieb,
{\em Exact Solution of the F Model of An Antiferroelectric}
Phys. Rev. Lett. vol. 18 pp. 1046-1048, 1967

\bibitem{Lieb673}
E. H. Lieb,
{\em Exact Solution of the Two-Dimensional Slater KDP Model of a Ferroelectric}
Phys. Rev. Lett. vol. 19 pp. 108-110, 1967


\bibitem{NC}
M. A. Nielsen and I. L. Chuang
{\em Quantum Computation and Quantum Information}
Cam. Uni. Press, 2010

\bibitem{Shor}
P. W. Shor,
{\em Algorithms for quantum computation: discrete logarithms and factoring}
Proc. 35th Ann. Symp. Found. Comp. Sci., pp. 124-134, 1994

\bibitem{Preskilletal}
C. Wang, J. Harrington, and J. Preskill,
{\em Confinement-Higgs transition in a disordered gauge theory and the
accuracy threshold for quantum memory}
Ann. Phys. 303 31, 2003

\bibitem{S}
B. Sutherland,
{\em Two-Dimensional Hydrogen Bonded Crystals without the Ice Rule}
J. Math. Phys. 11 3183, 1970

\end{thebibliography}
\end{document}